\documentclass{amsart}

\usepackage{rotating}
\usepackage{mathrsfs}
\usepackage{amsmath}
\usepackage{amssymb}
\usepackage{color}
\usepackage{enumerate}
\usepackage{amsthm}
\usepackage{bbm}
\usepackage{listings}
\usepackage[all]{xy}
\usepackage{longtable}
\usepackage{graphicx}
\usepackage{wrapfig}

\theoremstyle{definition}
\newtheorem{theorem}{Theorem}[section]
\newtheorem{proposition}[theorem]{Proposition}
\newtheorem{lemma}[theorem]{Lemma}
\newtheorem{corollary}[theorem]{Corollary}
\theoremstyle{definition}
\newtheorem{definition}[theorem]{Definition}
\theoremstyle{remark}
\newtheorem{remark}[theorem]{Remark}
\newtheorem{example}[theorem]{Example}

\newcommand{\qis}{\operatorname{qis}}
\newcommand{\Obs}{\operatorname{Obs}}

\newcommand{\Hom}{\operatorname{Hom}}

\newcommand{\Spec}{\operatorname{Spec}}
\newcommand{\D}{\operatorname{d}\!}
\newcommand{\CC}{\mathbb{C}}
\newcommand{\QQ}{\mathbb{Q}}
\newcommand{\ZZ}{\mathbb{Z}}

\newcommand{\RR}{\mathbb{R}}
\newcommand{\PP}{\mathbb{P}}
\newcommand{\OO}{\mathcal{O}}

\newcommand{\Grass}{\operatorname{Grass}}

\newcommand{\Jac}{\operatorname{Jac}}
\newcommand{\grad}{\operatorname{grad}}
\newcommand{\reg}{\operatorname{reg}}
\newcommand{\Eu}{\operatorname{Eu}}
\newcommand{\Ind}{\operatorname{Ind}}

\title{A generalization of Milnor's formula}
\author{Matthias Zach}
\address{Matthias Zach, Institut f\"ur Mathematik \\
FB 08 - Physik, Mathematik und Informatik \\
Johannes Gutenberg-Universit\"at Mainz\\
Staudingerweg 9\\
55128 Mainz, Germany}

\begin{document}

\begin{abstract}
  We describe a generalization of Milnor's 
  formula 
  \[
    \mu_f = \dim_\CC \OO_n/\Jac(f)
  \]
  for the 
  Milnor number $\mu_f$ of an isolated 
  hypersurface singularity 
  $f \colon (\CC^n,0) \to (\CC,0)$
  to the case of a function $f$ whose restriction $f|(X,0)$ 
  to an arbitrarily singular reduced complex analytic space 
  $(X,0) \subset (\CC^n,0)$ has an isolated singularity 
  in the stratified sense.
  The corresponding analogue of the Milnor 
  number, $\mu_f(\alpha;X,0)$, is the number of Morse critical points 
  in a stratum $\mathscr S_\alpha$ of $(X,0)$ in a morsification 
  of $f|(X,0)$.
  Our formula expresses 
  $\mu_f(\alpha;X,0)$ as a homological index based on the derived 
  geometry of the Nash modification of the closure 
  of the stratum.
  While most of the topological aspects in this setup were already 
  understood, our considerations provide the corresponding analytic 
  counterpart. 
  We also describe how to compute 
  the numbers
  $\mu_f(\alpha;X,0)$ by means of our formula in the case where 
  the closure $\overline{ \mathscr S_\alpha} \subset X$ of the stratum in question 
  is a hypersurface.
\end{abstract}

\maketitle

\tableofcontents

\section{Summary of results}
\label{sec:Summary}

We start by a discussion of the Milnor number similar to the one found in 
\cite{SeadeTibarVerjovsky05}. 
The Milnor number $\mu_f$ is one of the central invariants of 
a holomorphic function 
\[
  f \colon (\CC^n,0) \to (\CC,0)
\]
with isolated singularity. It has -- among others -- the following 
characterizations, cf. \cite[Chapter 7]{Milnor68} and 
\cite[Chapter 2]{ArnoldGuseinZadeVarchenkoVolII}.
\begin{enumerate}[1)]
  \item It is the number of Morse critical points in a morsification 
    $f_\eta$ of $f$. 
  \item It is equal to the middle Betti number of the Milnor fiber 
    \[
      M_f = B_\varepsilon \cap f^{-1}(\{\delta\}), \quad \varepsilon \gg \delta >0.
    \]
  \item It is the degree of the map 
    \[
      \frac{1}{|\D f|} \D f \colon \partial B_\varepsilon \to S^{2n-1}
    \]
    for some choice of a Hermitian metric on $(\CC^n,0)$.
  \item It is the length of the Milnor algebra
    \[
      \OO_n / \Jac(f), 
    \]
    where $\Jac(f) = 
    \langle \frac{\partial f}{ \partial x_1},\dots,\frac{\partial f}{\partial x_n} \rangle$ 
    is the Jacobian ideal of $f$.
\end{enumerate}

In this note we consider the more general setup of an arbitrary  
reduced complex analytic space $(X,0) \subset (\CC^n,0)$ and 
a holomorphic 
function $f \colon (\CC^n,0) \to (\CC,0)$, whose restriction $f|(X,0)$
to $(X,0)$ has an isolated singularity in the stratified sense. 
To this end, we will assume that $(X,0)$ is endowed with a complex 
analytic Whitney 
stratification $S = \{ \mathscr S_\alpha \}_{\alpha \in A}$
with finitely many connected strata $\mathscr S_\alpha$. 
There always exists a Milnor fibration for the restriction 
$f|(X,0)$ of any function $f$ to $(X,0)$, regardless of whether 
or not $f|(X,0)$ has isolated singularity; see \cite{Le87}, or 
\cite{GoreskyMacPherson88}. Denote the corresponding Milnor fiber by 
\[
  M_{f|(X,0)} = B_\varepsilon \cap X \cap f^{-1}(\{\delta\}),
\]
where $X$ is a suitable representative, $B_\varepsilon$ a ball of radius $\varepsilon$ 
centered at the origin in $\CC^n$, and $\varepsilon \gg \delta > 0$ sufficiently small.

\medskip
We introduce invariants $\mu_f(\alpha;X,0)$ of $f|(X,0)$ --
see Definition \ref{def:MuFAlpha} --
which generalize the classical Milnor number \textit{simultaneously}
in all of these four characterizations. Let $X_\alpha = \overline{ \mathscr S_\alpha}$ 
be the closure of the stratum $\mathscr S_\alpha$ and $d(\alpha)$ its (complex) 
dimension. Then for every $\alpha \in A$ 
the number $\mu_f(\alpha;X,0)$ is

\begin{enumerate}[1')]
  \item the number of Morse critical points on the stratum $\mathscr S_\alpha$ 
    in a morsification of $f$.
    For the definition of morsifications in this context 
    see Section \ref{sec:Morsifications}.
  \item the number of direct summands for $\alpha$ in the homology 
    decomposition of the Milnor fiber $M_{f|(X,0)}$, see 
    Proposition \ref{prp:HomologyDecompositionStratifiedMorsification}. 
  \item the Euler obstruction $\Eu^{\D f}\left( X_\alpha,0 \right)$
    of the $1$-form $\D f$ on $(X_\alpha,0)$, 
    see Definition \ref{def:EulerObstructionOfAOneForm} and Corollary
    \ref{cor:MuEqualsEuForFunctions}.
  \item the homological index
    \[
      \mu_f(\alpha;X,0) = (-1)^{d(\alpha)} \cdot 
      \chi\left( \mathbb R \nu_* \left(\tilde \Omega^\bullet_{\alpha}, 
      \nu^* \D f \wedge - \right)_0 \right),
    \]
    i.e. as an Euler characteristic of a finite complex of coherent 
    $\OO_X$-modules, cf. Theorem \ref{thm:MainTheorem} and Corollary 
    \ref{cor:MainCorollary}.
\end{enumerate}

Generalizations similar to those of 1), 2), and 3) have been made
by J. Seade, M. Tib{\u a}r and A. Verjovsky
in \cite{SeadeTibarVerjovsky05}.
The Euler obstruction of a $1$-form was introduced by W. Ebeling and 
S. Gusein-Zade in \cite{EbelingGuseinZade05}. 
Contrary to these previous topological considerations, we will describe the 
Euler obstructions $\Eu^{\D f}(X_\alpha,0)$ 
as an \textit{analytic} invariant in 
Theorem \ref{thm:MainTheorem}. 
This allows us to also generalize the characterization 4) of the Milnor number 
to 4'). A description for how to compute the numbers $\mu_f(\alpha;X,0)$ whenever 
$\overline{\mathscr S_\alpha}$ is an algebraic hypersurface, $f$ is also 
algebraic and both are defined over a finite extension field of $\QQ$,
is described in Section \ref{sec:Algorithm}.

\begin{figure}[h]
  \centering
  \includegraphics[scale=0.14]{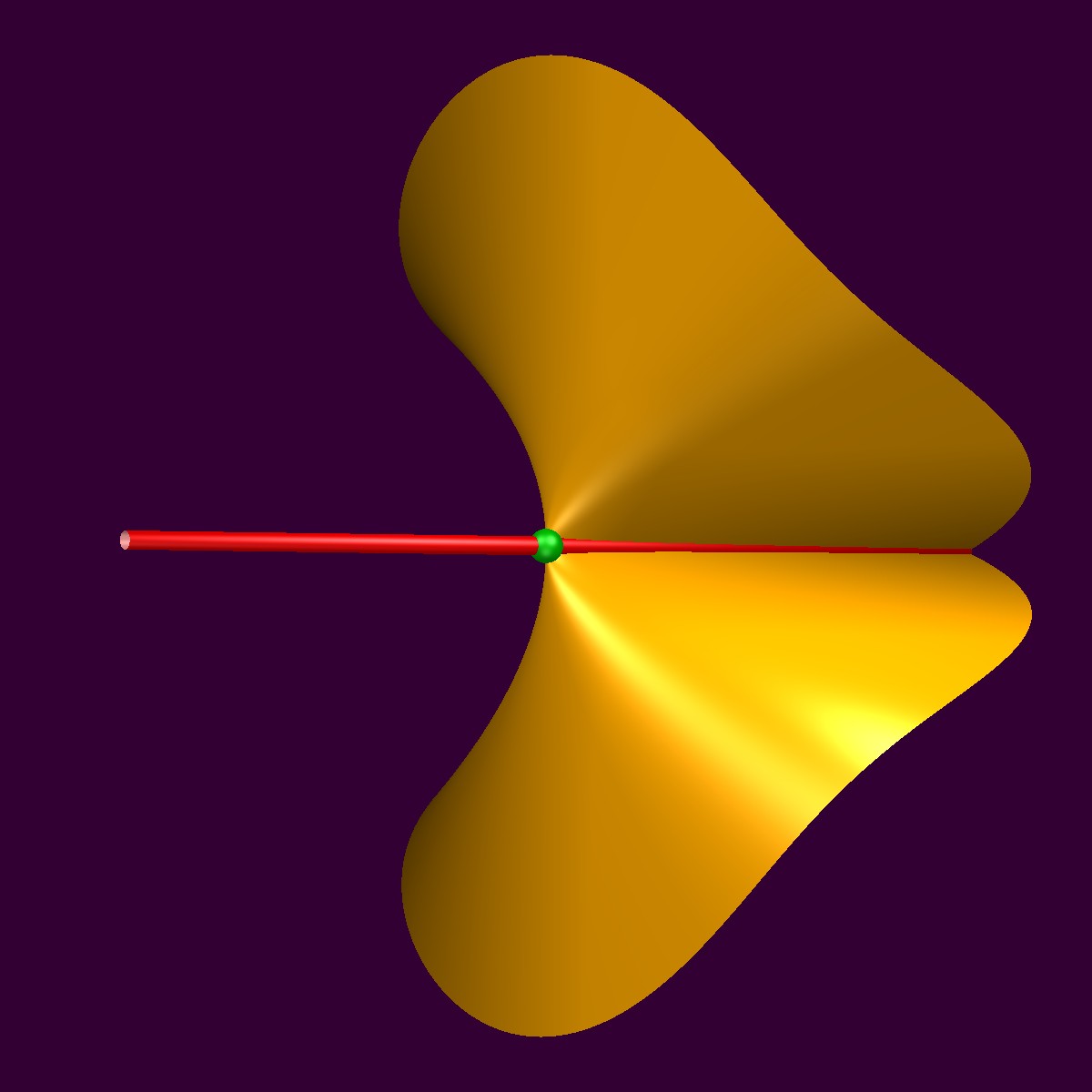}
  \includegraphics[scale=0.14]{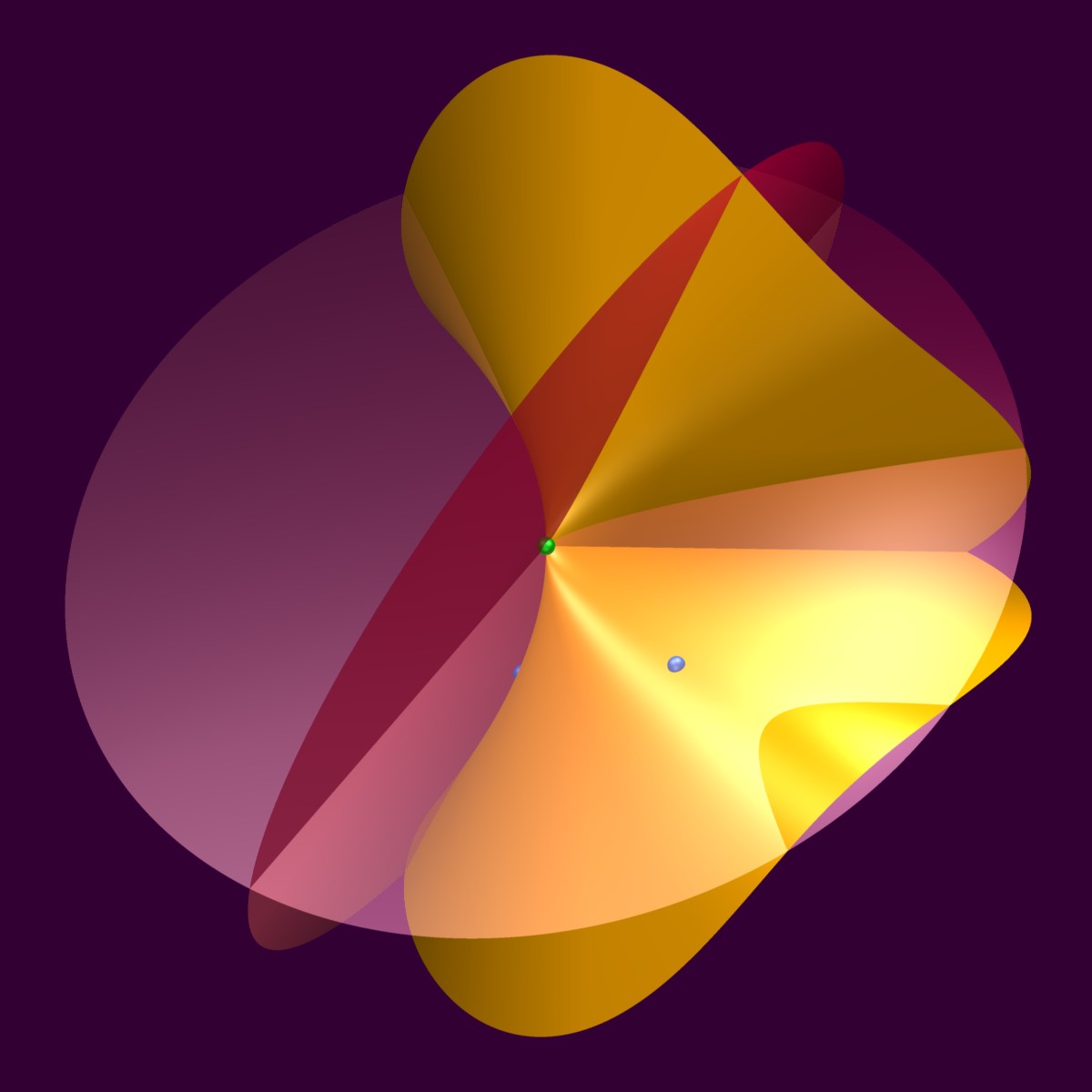}\\
  \vspace{1mm}
  \includegraphics[scale=0.14]{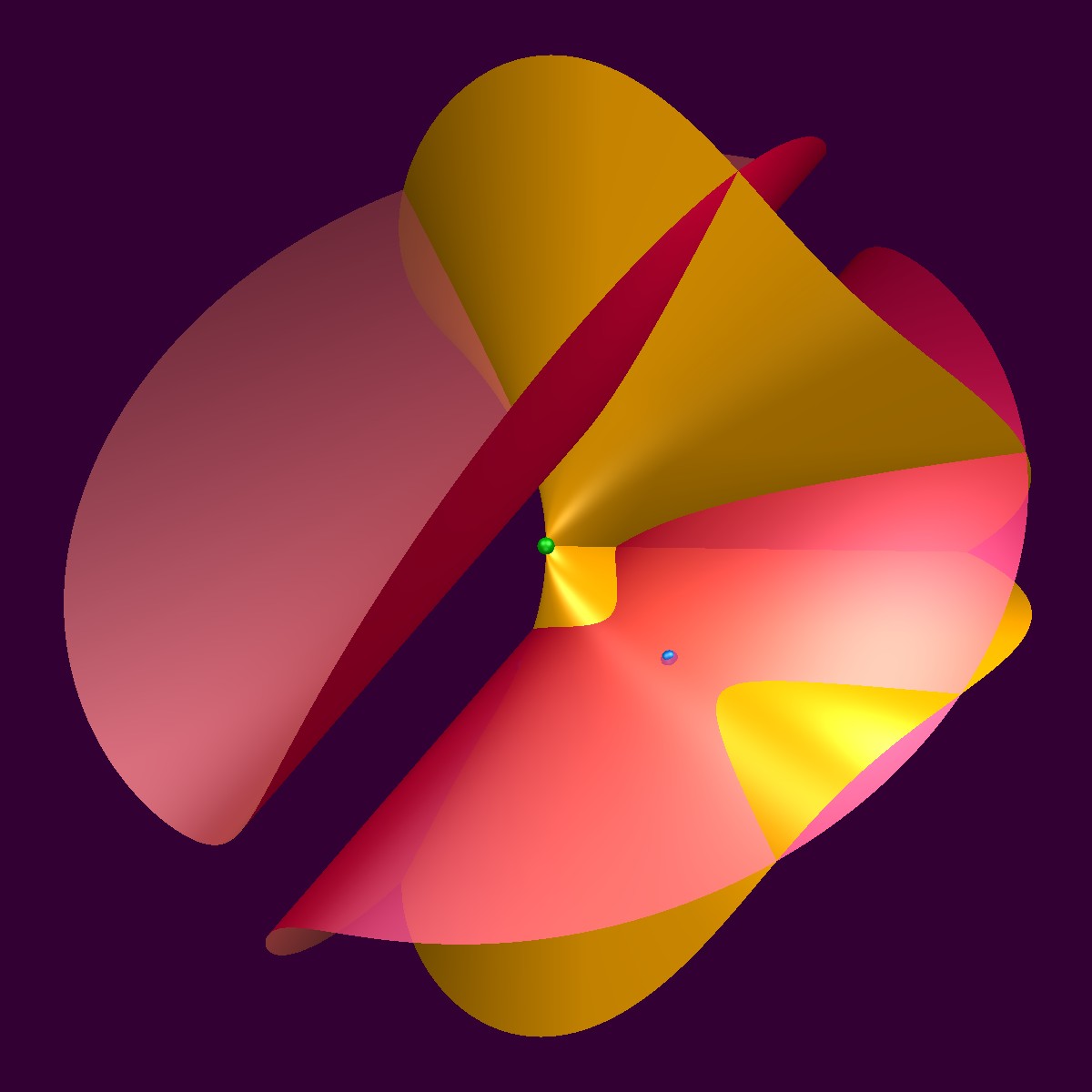}
  \includegraphics[scale=0.14]{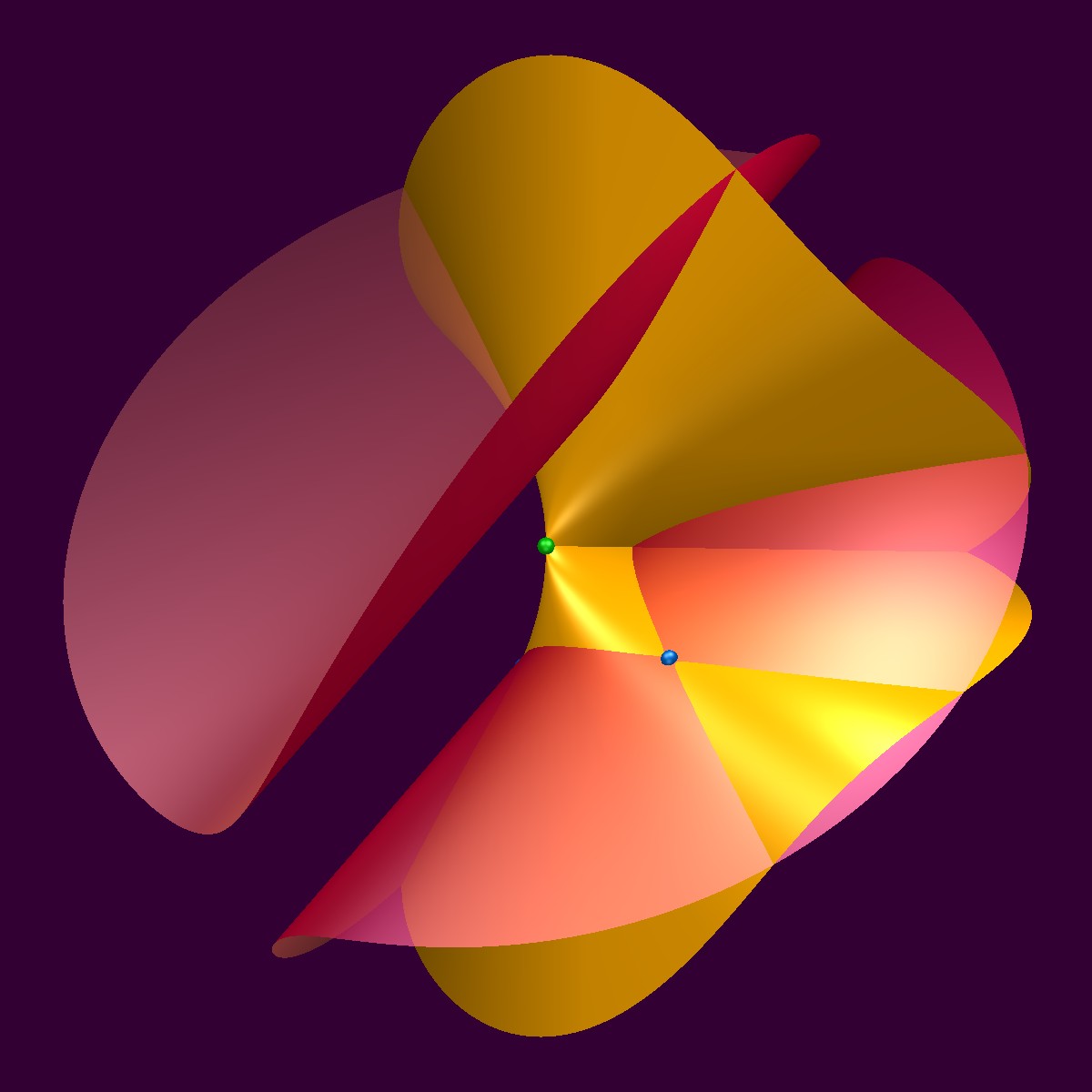}

  \caption{The Whitney Umbrella with 1) its three strata, 2) the zero level 
    of $f$ and the critical points of $f|X$, and level sets of $f$ of 3) a regular 
    value of $f|X$, and 4) a critical value of $f|X$.
  }
  \label{fig:WhitneyUmbrellaThreeStrata}
\end{figure}

\begin{example}
  \label{exp:Introduction}
  The following will serve us as a guiding example throughout this article.
  Let $X \subset \CC^3$ be the Whitney umbrella given by 
  the equation $h = y^2-xz^2 = 0$ and endowed with the stratification 
  \begin{eqnarray*}
    \mathscr S_0 &=& \{0\},\\
    \mathscr S_1 &=& \{y = z = 0 \} \setminus \mathscr S_0,\\
    \mathscr S_2 &=& X \setminus (\mathscr S_0 \cup \mathscr S_1 ).
  \end{eqnarray*}
  This stratification is known to satisfy the Whitney conditions A and B.

  As a function $f \colon \CC^3 \to \CC$ with isolated singularity on $(X,0)$ we consider 
  \[
    f(x,y,z) = y^2 - (x-z)^2.
  \]
  Note that $f$ does \textit{not} have isolated singularity on $\CC^3$.
  Its restriction $f|X$ to $X$, however, has only isolated critical points at 
  \[
    0 = \begin{pmatrix}
      0, & 0, & 0 
    \end{pmatrix} \quad 
    \text{and} \quad
    p_{6,7} = \begin{pmatrix}
      \frac{3}{2}, & \pm \frac{3}{\sqrt{2}}, & -3
    \end{pmatrix}.
  \]
  It will become clear later, why we label the last two of these points with indices $6$ and 
  $7$. We will usually have to neglect these points, since
  we are interested in the local behaviour of $f$ on the germ $(X,0)$ 
  of $X$ at the origin $0 \in \CC^3$. 
  As we shall see in Example \ref{exp:Morsification}, 
  we have 
  \[
    \mu_f(0;X,0) = 1, \quad \mu_f(1;X,0) = 1, \quad \mu_f(2;X,0) = 5 
  \]
  for the restriction $f|(X,0)$ at this point.
\end{example}

\section*{Acknowledgements}

The author is indebted to J. Seade for discussions on the homological 
index, the law of conservation of number, and Euler obstructions. 
These conversations took place at the meeting ENSINO V at 
Jo\~ao Pessoa, Brazil in July 2019 and the author wishes to also thank the 
organizers A. Menegon Neto, J. Snoussi, M. da Silva Pereira, and M. F. Zanchetta Morgado 
for the opportunity to give a course on determinantal 
singularities together with M.A.S. Ruas. The author would also like to thank 
J. Sch\"urmann for encouraging discussions during the conference ``Non-isolated singularities 
and derived geometry'' at Cuernavaca, Mexico the same year. It is the ideas 
from this period that cumulated in this article. The pictures for the examples 
were produced using the software ``Surfer'' by O. Labs et. al.

\section{Background and motivation}
\label{sec:Background}

Suppose the function $f$, 
the space $(X,0)$, and its stratification have been chosen as in 1') to 4') 
from 
Section \ref{sec:Summary}.
In \cite{SeadeTibarVerjovsky05}
the Euler obstruction $\Eu_f(X,0)$ of the function $f$ on $(X,0)$ plays 
the role of the $\mu_f(\alpha;X,0)$ for the top-dimensional stratum -- up to sign. 
The Euler obstruction of a function was introduced in 
\cite{BrasseletMasseyParameswaranSeade04} and it is defined as follows. 
Let $\nu \colon \tilde X \to X$ be the 
Nash modification of $(X,0)$. Then there always exists a continuous alteration $v$ 
of the gradient vector field $\grad f$ on $(\CC^n,0)$ which is tangent to
the strata of $(X,0)$, and a lift 
$\nu^* v$ of $v$ to the Nash bundle 
$\tilde T$ on $\tilde X$. Over the link $K = \partial B_\varepsilon \cap X$ of 
$(X,0)$ this lift is well defined as a non-zero section in $\tilde T$ up to 
homotopy. Now $\Eu_f(X,0)$ is the obstruction to extending $\nu^* v$ as 
a nowhere vanishing section to the interior of $\tilde X$.

To understand how our approach came about to also include 4') in this discussion,
we have to consider the article \cite{SeadeTibarVerjovsky05} in the context of a 
series of articles by various authors  
on different \textit{indices} of vector fields and $1$-forms on 
singular varieties. A thorough survey of the results from that time is 
\cite{EbelingGuseinZade06}.

One of these indices -- the GSV index of a vector field -- 
is particularly close to the idea of the Euler obstruction. 
The GSV index was first defined in 
\cite[Definition 2.1 ii)]{GomezMontSeadeVerjovsky91} for the following setup:

Let 
$(X,0) = (g^{-1}(\{0\}),0) \subset (\CC^{n+1},0)$
be an isolated hypersurface singularity 
and $v$ the germ of a 
vector field on $\left( \CC^{n+1},0 \right)$ which has an isolated zero 
at the origin and is tangent to 
$(X,0)$.
The GSV index $\Ind_{GSV}(v,X,0)$ of $v$ on $(X,0)$ is the obstruction 
to extending the section $v|K$ as a $C^\infty$-section of the 
tangent bundle from the link $K = X\cap \partial B_\varepsilon$ 
to the interior of the Milnor fiber $B_\varepsilon \cap g^{-1}(\{\delta\})$.
Here we deliberately identify the link $K$ with the boundary 
$\partial B_\varepsilon \cap g^{-1}(\{\delta\})$ of the Milnor fiber and 
the section $v|K$ with its image under this identification. 

In \cite{GomezMont98}, X. G\'omez-Mont introduces the \textit{homological index}
of a vector field $v$ on $(X,0)$ as above in order to compute the GSV index 
algebraically\footnote{Algebraic formulae for the GSV index of $v$ on $(X,0)$ 
  were also given in \cite{GomezMontSeadeVerjovsky91}, but 
only under the assumption that $v$ was also tangent to all fibers of $f$.}. 
It is defined as
$\Ind_{\hom}(v,X,0) = \chi(\Omega^\bullet_{X,0}, v)$, i.e. 
the Euler characteristic 
of the complex 
\[
  \xymatrix{
    0 & 
    \OO_{X,0} \ar[l] & 
    \Omega^1_{X,0} \ar[l]_v &
    \Omega^2_{X,0} \ar[l]_v &
    \cdots \ar[l]_v &
    \Omega^{n-1}_{X,0} \ar[l]_v &
    \Omega^{n}_{X,0} \ar[l]_v &
    0 \ar[l]_v
  }
\]
where $\Omega^p_{X,0}$ denotes the module of universally finite K\"ahler differentials 
on $(X,0)$ and $v$ is the homomorphism given by contraction of a 
differential form with the vector field $v$. Later on in his article, 
X. Gomez-Mont generalizes the GSV index in the obvious way to the setting 
of an arbitrary 
complex space $(X,0)$ with an isolated singularity and a fixed 
smoothing $X'$ of $(X,0)$. 
In \cite[Theorem 3.2]{GomezMont98} he proves that 
\begin{equation}
  \Ind_{GSV}(v,X,X') - \Ind_{\hom}(v,X,0) = k(X,X')
  \label{eqn:GSVvsHomologicalIndex}
\end{equation}
with $k(X,X')$ a constant depending only on $(X,0)$ and the chosen smoothing, 
i.e. independent of the vector field $v$. Finally, he shows in 
\cite[Section 3.2]{GomezMont98} that 
whenever $(X,0) = (g^{-1}(\{0\}),0)$ is an isolated hypersurface singularity 
with its canonical smoothing $X' = B_\varepsilon \cap g^{-1}(\{\delta\})$, 
$\varepsilon \gg \delta >0$, then $k(X,X') = 0$. 

From our point of view, the main novum in the approach by X. G\'omez-Mont 
was the introduction of derived geometry in this setting and its comparison 
with topological invariants. To prove Equation (\ref{eqn:GSVvsHomologicalIndex}), 
he procedes as follows. 

On the one hand, the GSV index is constant under small perturbations of $v$. 
This is immediate from the definition, since small perturbations of $v$ do not 
change the homotopy class of the non-zero section $v|K$. 
On the other hand, the homological index $\Ind_{\hom}(v,X,0)$ satisfies the 
\textit{law of conservation of number}, i.e. for suitable representatives 
and $\tilde v$ sufficiently close to $v$ one has 
\[
  \Ind_{\hom}(v,X,0) = \sum_{p\in X} \Ind_{\hom}(\tilde v, X,p).
\]
This is due to a technical but fundamental result 
based on derived geometry for the complex analytic setting from 
\cite{GiraldoGomezMont02} which states that, more generally, 
the Euler characteristic of a complex of coherent 
sheaves with finite dimensional cohomology satisfies the law of 
conservation of number.
To conclude the proof of Equation (\ref{eqn:GSVvsHomologicalIndex}), 
observe that at smooth points $p \in X_{\reg}$, the GSV index and the homological 
index coincide. Since the space of holomorphic vector fields on $(X,0)$ with 
isolated singularity at the origin is connected, 
the difference $\Ind_{GSV}(v,X,X') - \Ind_{\hom}(v,X,0)$ must be a 
constant $k(X,X')$ and in particular 
independent of the vector field $v$.

\medskip

In this article, we will not be dealing with vector fields, but with 
holomorphic $1$-forms. In fact, the original definition of the 
Euler obstruction by R. MacPherson in \cite{MacPherson74} was phrased 
in terms of radial $1$-forms and only later the use of vector fields 
became popular following the work of J.P. Brasselet and M.H. Schwartz 
\cite{BrasseletSchwartz81}.
The use of $1$-forms is more natural in the context of morsifications 
and it has several further advantages. For example, we can drop the
tangency conditions to $(X,0)$ which we had to impose on any vector field 
$v$.

It is straightforward -- and even easier --
to also define the Euler obstruction $\Eu^\omega(X,0)$ of 
a $1$-form $\omega$ with isolated zero on $(X,0)$: 
Again, let $\nu \colon \tilde X \to X$ be the Nash modification. 
Then there is a natural pullback $\nu^* \omega$ of $\omega$ to 
a section of the \textit{dual} of the Nash bundle and this section 
does not vanish 
on $\nu^{-1}(\partial B_\varepsilon \cap X)$ whenever $\omega$ 
has an isolated zero on $(X,0)$ in the stratified sense. 
The Euler obstruction of such an $\omega$ on $(X,0)$ is the obstruction 
to extending $\nu^* \omega$ as a nowhere vanishing section to the 
interior of $\tilde X$.

There is a natural notion of the homological index for a $1$-form 
$\omega$ with isolated zero on any purely $n$-dimensional complex analytic space $(X,0)$
with \textit{isolated} singularity. 
In \cite{EbelingGuseinZadeSeade04}, W. Ebeling, S.M. Gusein-Zade, and J. Seade define 
\[
  \Ind_{\hom}(\omega,X,0) = (-1)^n \chi(\Omega^\bullet_{X,0}, \omega\wedge - )
\]
where $\left( \Omega^\bullet_{X,0}, \omega \wedge - \right)$ is the complex
\begin{equation}
  \xymatrix{
    0 \ar[r] & 
    \OO_{X,0} \ar[r]^{\omega \wedge} & 
    \Omega^1_{X,0} \ar[r]^{\omega \wedge} & 
    \cdots \ar[r]^{\omega \wedge} & 
    \Omega^{n-1}_{X,0} \ar[r]^{\omega \wedge} & 
    \Omega^n_{X,0} \ar[r] & 
    0
  }
  \label{eqn:ComplexWedgeWithOmega}
\end{equation}
with differential given by the exterior multiplication with $\omega$.
Note that in case $(X,0) \cong (\CC^n,0)$ is smooth and $\omega = \D f$ 
is the differential of a function $f$ with isolated singularity on $(X,0)$, 
the homological index coincides with the classical Milnor number. 
This is due to the fact that the complex (\ref{eqn:ComplexWedgeWithOmega}) 
is the Koszul complex in the partial derivatives $\frac{\partial f}{\partial x_i}$ 
of $f$ which is known to be a free resolution of the Milnor algebra for 
an isolated hypersurface singularity.

When $(X,0)$ has isolated singularity, there is no immediate interpretation 
for the homological index of $\omega$ in terms of previously known 
invariants. However, it is relatively easy 
to see with the same reasoning as for indices of vector fields 
that the the difference 
\begin{equation}
  \Eu^{\omega}(X,0) - \Ind_{\hom}(\omega,X,0) = k'(X,0)
  \label{eqn:EulerObstructionVsHomologicalIndex}
\end{equation}
is also a constant, independent of the $1$-form $\omega$: 
The Euler obstruction $\Eu^{\omega}(X,0)$ is a homotopy invariant 
and $\Ind_{\hom}(\omega,X,0)$ satisfies 
the law of conservation of number. Suppose we have chosen a suitable 
representative $X$ of $(X,0)$ and a sufficiently small ball $B_\varepsilon$. 
Then for any a holomorphic $1$-form $\omega'$ on $X$ which has only isolated zeroes 
on the smooth part $X_{\reg}$ of $X$ and which is sufficiently close to the original $1$-form 
$\omega$, we have 
\begin{eqnarray*}
  & & \Eu^{\omega}(X,0) - \Ind_{\hom}( \omega, X,0) \\
  &=&  \sum_{p \in X \cap B_\varepsilon} \Eu^{\omega'}(X,p)  
  - \sum_{p \in X \cap B_\varepsilon} \Ind_{\hom}(\omega',X,p) \\
  &=& \Eu^{\omega'}(X,0) - \Ind_{\hom}(\omega',X,0) 
  + \sum_{p \in X_{\reg} \cap B_\varepsilon}
  \left( \Eu^{\omega'}(X,p) - \Ind_{\hom}(\omega',X,p) \right) \\
  &=& \Eu^{\omega'}(X,0) - \Ind_{\hom}(\omega',X,0). 
\end{eqnarray*}
This holds because, again, $\Eu^{\omega'}(X,p) = \Ind_{\hom}(\omega', X,p)$ at 
smooth points $p \in X_{\reg}$. The general claim now follows from the 
fact that the set of those holomorphic $1$-forms on $X$ with only 
only isolated zeroes on $X_{\reg}$ is open and connected. 

\medskip

There are other instances of very similar discussions. 
In \cite[Proposition 4.1]{EbelingGuseinZadeSeade04}, for example, there is a comparison 
of the homological index and the \textit{radial index} $\Ind_{\operatorname{rad}}(\omega,X,0)$
(cf. \cite[Definition 2.1]{EbelingGuseinZadeSeade04})
of a $1$-form $\omega$ with isolated zero on an equidimensional complex analytic space 
$(X,0)$ with isolated singularity. Their difference is an 
invariant $\nu(X,0)$ which coincides with the Milnor number 
of $(X,0)$ whenever $(X,0)$ is an isolated complete intersection 
singularity.

For the same setting there is another comparison 
of the radial index and the 
Euler obstruction $\Eu^{\omega}(X,0)$ 
in \cite[Proposition 4]{EbelingGuseinZade05}. In this case
\begin{equation}
  \Eu^{\omega}(X,0) 
  - \Ind_{\operatorname{rad}}(\omega,X,0) 
  = (-1)^{\dim(X,0)} \overline \chi ( \mathcal L(X,0) ),
  \label{eqn:RadialIndexVsEulerObstruction}
\end{equation}
where $\mathcal L(X,0)$ denotes the complex link of $(X,0)$ 
(see e.g. \cite{GoreskyMacPherson88}) and $\overline \chi$ 
is the reduced topological Euler characteristic.

\medskip

Coming back to the comparison of $\Eu^{\omega}(X,0)$ with 
$\Ind_{\hom}(\omega,X,0)$ in Equation (\ref{eqn:EulerObstructionVsHomologicalIndex}), 
the introduction of 
$k'(X,0)$ as a new invariant of the germ $(X,0)$ seems to be rather 
unmotivated. Instead
we propose 
a modification of the homological index 
in Section \ref{sec:TheMilnorNumberAsAHomologicalIndex}
which directly computes the Euler obstruction. 
This is Theorem \ref{thm:MainTheorem}.
The new homological index will be based on the Nash modification 
$\nu \colon \tilde X \to X$ of $(X,0)$ and the complex of sheaves 
$\left( \tilde \Omega^\bullet, \nu^* \D \omega \wedge - \right)$ 
on $\tilde X$ rather than 
$\left(\Omega_{X,0}^\bullet, \omega \wedge - \right)$. For 
the definition of this complex see Sections \ref{sec:EulerObstruction} and 
\ref{sec:TheMilnorNumberAsAHomologicalIndex}.
The direct computation of $\Eu^{\omega}(X,0)$ as an 
Euler characteristic of finite $\OO_n$-modules comes  
at the price that one has to take the derived pushforward along 
$\nu$ of the complex of sheaves 
$\left( \tilde \Omega^\bullet, \nu^* \D \omega \wedge - \right)$.
However, as a side effect of this, we may drop the 
assumption on $(X,0)$ to have only isolated singularity.

\section{Generalizations of the Milnor number}
\label{sec:GeneralizationsOfTheMilnorNumber}

We briefly recall the necessary definitions of singularity 
theory on stratified spaces, cf. \cite{Le87}.
Let $U \subset \CC^n$ be an open domain, $X \subset U$ a 
closed, reduced, equidimensional 
complex analytic set and $f \colon U \to \CC$ a holomorphic function. 

\begin{definition}
  \label{def:StratifiedRegularPoint}
  Suppose $S = \{ \mathscr S_\alpha \}_{\alpha \in A}$ is a complex analytic 
  Whitney stratification of $X$. 
  A point $p \in X$ is a \textit{regular point} 
  of $f|X$ in the stratified sense, if it is a regular point of the restriction 
  $f|\mathscr S_\alpha$ of $f$ to the stratum $\mathscr S_\alpha$ containing $p$.
\end{definition}

The existence of complex analytic Whitney stratifications was shown by H. Hironaka \cite{Hironaka77}.
In \cite[Corollaire 6.1.8]{LeTeissier81} L\^e D.\,T. 
and B. Teissier constructed a canonical Whitney stratification 
for reduced, equidimensional complex analytic spaces, and in \cite{Teissier82} it was 
shown that this stratification is minimal. Whenever one of these strata consists of several 
components, we shall in the following consider each one of these components as a 
stratum of its own and -- unless otherwise specified -- 
use this stratification on any given reduced equidimensional 
complex analytic space $X$.

\begin{definition}
  \label{def:IsolatedSingularityOfAFunction}
  We say that $f$ has an \textit{isolated singularity} at $(X,p)$, 
  if there exists a neighborhood $U'$ of $p$ such that all points $x \in U' \cap X \setminus \{p\}$ 
  are regular points of $f$ in the stratified sense for the canonical Whitney stratification 
  of $X$.
\end{definition}

We give a brief definition of the Milnor fibration of $f|(X,p)$ in this 
setting. Let $B_\varepsilon$ be the ball of radius $\varepsilon$ around $p$ 
in $\CC^n$. By virtue of the Curve Selection Lemma, there exists $\varepsilon_0>0$ 
such that for every $\varepsilon_0 \geq \varepsilon >0$ the intersections 
$\partial B_\varepsilon \cap X$ and $\partial B_\varepsilon \cap X \cap f^{-1}(\{f(p)\})$ 
are transversal. Fix one such $\varepsilon>0$. Then for sufficiently small 
$\varepsilon \gg \delta >0$ the restriction of $f$ 
\begin{equation}
  f \colon B_\varepsilon \cap X \cap f^{-1}(D^*_\delta) \to D^*_\delta
  \label{eqn:MilnorFibration}
\end{equation}
is a proper $C^0$-fiber bundle 
over the punctured disc $D_\delta^* \subset \CC$ of radius $\delta>0$ around $f(p)$ -- 
the Milnor fibration of $f|(X,p)$. The fiber 
\[
  M_{f|(X,p)} = B_\varepsilon \cap X \cap f^{-1}(\{\delta\})
\]
is unique up to homeomorphism and thus an invariant of $f|(X,p)$. 

\begin{remark}
  Many authors prefer to work with a holomorphic function $g \colon (X,p) \to (\CC,g(p))$ 
  instead of an embedding $\iota \colon (X,p) \hookrightarrow (\CC^n,p)$ and a restriction 
  $f|(X,p)$ of a function $f \colon (\CC^n,p) \to (\CC,f(p))$. It is clear that 
  for every $g$ one can find $\iota$ and $f$ such that $f|(X,p) = g$. 
  Moreover, 
  it can be shown that the canonical stratification of $(X,p)$ does not depend 
  on the embedding \cite{Le87}. Neither does the Milnor fiber $M_g = M_{f|(X,p)}$.
  If the reader intends to start from $g$ defined on $(X,0)$, 
  he/she is supposed to make the necessary translations throughout the rest of the 
  article.
\end{remark}

\subsection{Morsifications}
\label{sec:Morsifications}
For functions on stratified spaces the most simple singularities are the 
\textit{stratified} Morse critical points.
They generalize the classical Morse critical points 
of a holomorphic function in the sense that every function $f$ with an isolated 
singularity on $(X,p)$ can be deformed to a function with finitely many stratified 
Morse critical points on $X$, cf. Corollary \ref{cor:ExistenceOfMorsifications}. 
Thus, they are the basic building blocks for the study of isolated singularities 
on stratified spaces. 

\begin{definition}(cf. \cite[Section 2.1]{GoreskyMacPherson88})
  \label{def:StratifiedMorseCriticalPoint}
  A point $p \in \mathscr S_\alpha \subset X$ is a stratified Morse critical 
  point of $f|X$ if 
  \begin{enumerate}[i)]
    \item the point $p$ is a classical Morse critical point of the restriction 
      $f|\mathscr S_\alpha$ of $f$ to the stratum $\mathscr S_\alpha$. 
    \item the differential $\D f(p)$ of $f$ at $p \in \CC^n$ does not 
      annihilate any limiting 
      tangent spaces $T \subset T_p \CC^n$ from other adjacent strata 
      $\mathscr S_\beta$ of $X$ at $p$.
  \end{enumerate}
\end{definition}

Consider a point $p \in U$ and the germ $f \colon (\CC^n,p) \to (\CC,f(p))$ 
of $f$ at $p$. 
An \textit{unfolding} of $f$ at $p$ is a map germ 
\[
  F \colon (\CC^n \times \CC, (p,0)) \to (\CC \times \CC,(f(p),0)), \quad 
  (x,t) \mapsto (f_t(x),t)
\]
such that $f= f_0$. It is clear that whenever $p\in X$,
any unfolding of $f$ induces an unfolding $F|(X,p)$ of $f|(X,p)$.

\begin{definition}
  \label{def:Morsification}
  Let $(X,p) \subset (\CC^n,p)$ be a reduced complex analytic 
  space and $f \colon (\CC^n,p) \to (\CC,f(p))$ a holomorphic function 
  with isolated singularity on $(X,p)$.
  An unfolding $F$ of $f$ induces a \textit{morsification} of $f|(X,p)$, if 
  there exists an open neighborhoods $V \subset \CC^n$ of $p$ 
  and an open disc $T \subset \CC$ around the origin 
  such that $f_t|X$ has only Morse critical 
  points in $X \cap V$ for all $0 \neq t \in T$.
\end{definition}

For the existence of morsifications and the density of Morse 
functions in the stratified setting see for example \cite{GoreskyMacPherson88}. 
We will 
usually take $f_t(x) = f(x) - t \cdot l(x)$ 
for a generic linear form $l \in \Hom(\CC^n,\CC)$, 
cf. Corollary \ref{cor:ExistenceOfMorsifications}. 

\medskip
We may choose the open neighborhood $V$ in Definition \ref{def:Morsification}
to be an open Milnor ball 
$B_\varepsilon$ for $f|(X,p)$. Then for $t = \eta \neq 0$ sufficiently small, 
all Morse critical points of $f_\eta$ on $X \cap B_\varepsilon$ arise 
from the original singularity of $f_0$ at $0 \in X$ and we can 
count the number of Morse critical points of $f_\eta$ on each 
stratum $\mathscr S_\alpha$ in $X \cap B_\varepsilon$.

\begin{definition}
  \label{def:MuFAlpha}
  We define the numbers $\mu_f(\alpha;X,0)$ of $f|(X,p)$ 
  to be the number of Morse critical 
  points on the stratum $\mathscr S_\alpha$ in a morsification of $f|(X,p)$.
\end{definition}

These numbers clearly depend on the choice of the stratification. However, 
it follows from \cite[Proposition 2.3]{SeadeTibarVerjovsky05}, that they 
do not depend on the choice of the morsification $F|(X,p)$ of $f|(X,p)$. 
This fact will also be a consequence of Theorem \ref{thm:MainTheorem}.
  
\begin{example}
  \label{exp:Morsification}
  We continue with Example \ref{exp:Introduction}.
  As a morsification of $f|(X,0)$ we may choose
  \[
    f_t (x,y,z) = y^2 - (x-z)^2 - t(x+2z).
  \]
  Clearly, $\mu_f(0;X,0) = 1$, because $\mathscr S_0$ is a one-point stratum and 
  any such point is a critical point of a function $f$ in the stratified 
  sense.

  \begin{figure}[h]
    \centering
    \includegraphics[scale=0.14]{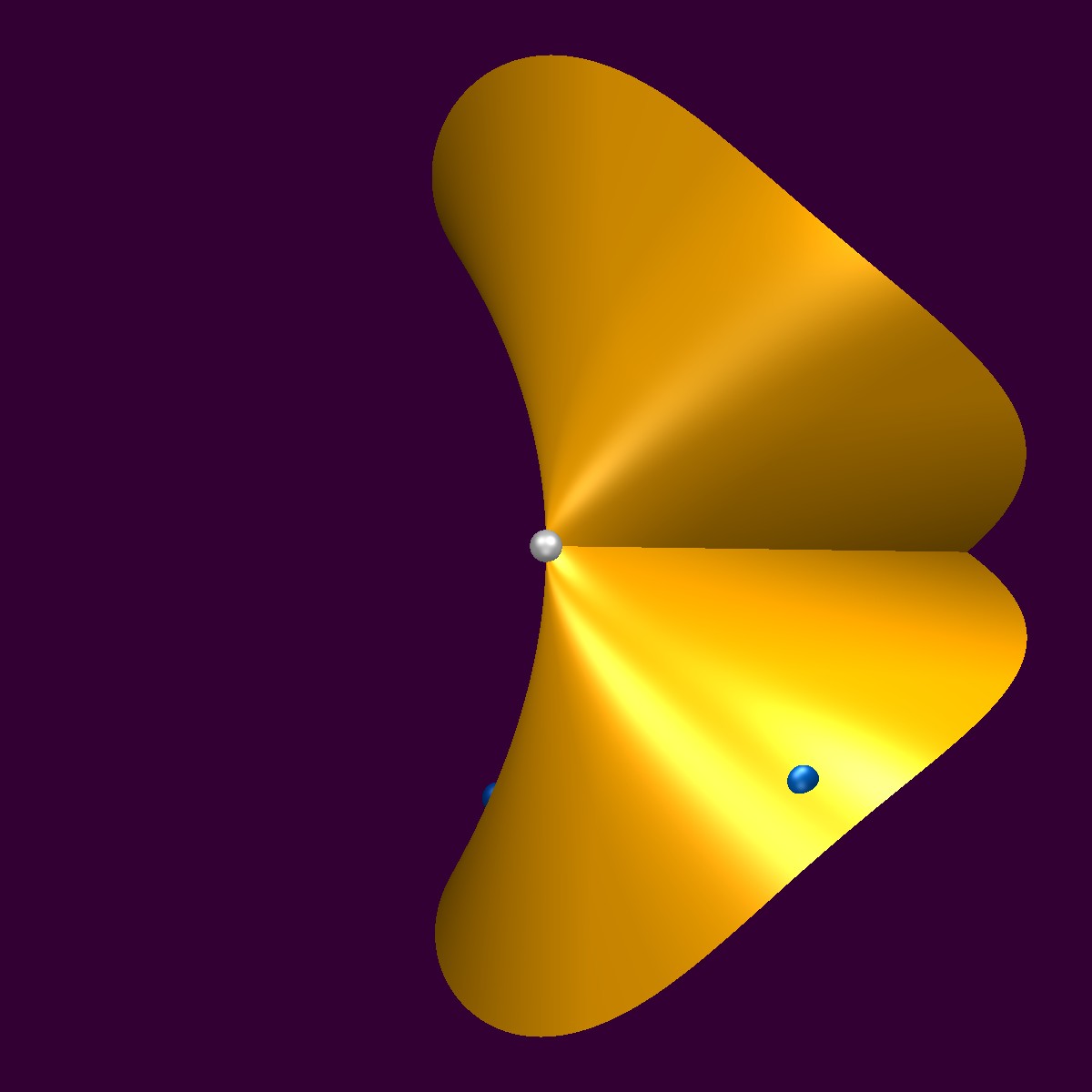}
    \includegraphics[scale=0.14]{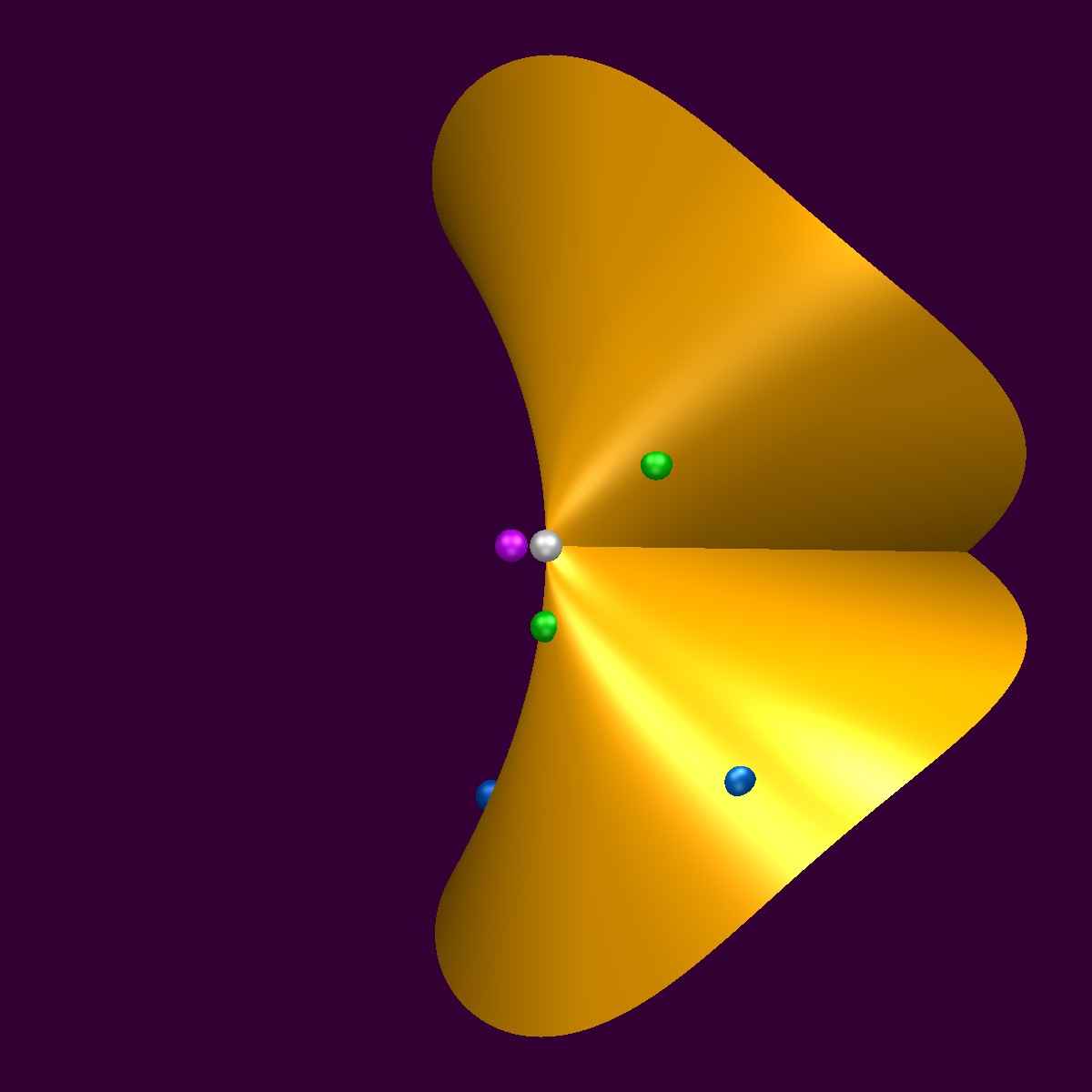}
    \caption{Closeup of the critical points of $f|X$ on the left 
    and those of $f_\eta|X$ for $\eta=1$ on the right hand side.}
    \label{fig:ZoomCriticalPoints}
  \end{figure}

  On $\mathscr S_1$ the function $f_\eta$ has exactly one 
  Morse critical point for $\eta \neq 0$. 
  This can be verified by classical methods: Note that 
  $X_1 = \overline{\mathscr S_1}$ is smooth and the restriction of $f$ to $X_1$ 
  is an ordinary $A_1$ singularity. The given morsification is moving this critical 
  point -- depicted in purple in Figure \ref{fig:ZoomCriticalPoints} --
  from $x=0$ to $x=-t/2$ so that for $t\neq 0$ it really lies in the stratum 
  $\mathscr S_1$.

  In order to compute $\mu_f(2;X,0)$ let 
  \[
    \Gamma = \overline{\{ (x,t) \in X_{\reg} \times \CC : 
      x \text{ is a critical point of $f_t$ on $X_{\reg}$ } \}}
  \]
  be the \textit{global} curve of critical points 
  of $f_t$ on the regular part $X_{\reg} = \mathscr S_2$ of the 
  whole affine variety $X \subset \CC^3$.
  Using a computer algebra system, one can verify that $\Gamma$ has seven branches. 
  Five of these branches 
  \[
    \Gamma_1(t) =
    \begin{pmatrix}
      0, \\
      0, \\ 
      -t
    \end{pmatrix}, \qquad
    \Gamma_{2,3}(t) =
    \begin{pmatrix}
      \sqrt{t}\\
      \pm \sqrt[4]{t^3} \\
      \sqrt{t}
    \end{pmatrix}, \qquad 
    \Gamma_{4,5} = 
    \begin{pmatrix}
      -\sqrt{t}\\
      \pm i\sqrt[4]{t^3} \\
      -\sqrt{t}
    \end{pmatrix}
  \]
  pass through the origin $0 \in \CC^3$, i.e. they arise 
  from the critical point of $f$ on $(X,0)$. Note that $\Gamma_{4,5}(t)$ does not 
  have real coordinates for $t \in \RR \setminus\{0\}$, so we will not be able to 
  illustrate these branches in real pictures. Nevertheless, the behaviour of 
  $\Gamma_{4,5}(t)$ is symmetric to what happens with the real branches $\Gamma_{2,3}(t)$.
  Each one of these branches corresponds to a Morse critical 
  point of $f_t$ on $\mathscr S_2 \subset X$ and we drew them as 
  green dots in the picture on the right of 
  Figure \ref{fig:ZoomCriticalPoints}. Thus we have
  \[
    \mu_f(0;X,0) = 1, \quad \mu_f(1;X,0) = 1, \quad \mu_f(2;X,0) = 5.
  \]
  
  The remaining two branches 
  \[
    \Gamma_{6,7}(t) =
    \begin{pmatrix}
      \frac{3}{2} - \frac{t}{2}\\
      \pm \sqrt{\frac{27 - 9t}{2}}\\
      -3
    \end{pmatrix}
  \]
  are swept out from the points $p_6$ and $p_7$ and do not contribute to the 
  number $\mu_f(2;X,0)$ of $f|(X,0)$ at the origin. They correspond to the 
  blue dots in Figure \ref{fig:ZoomCriticalPoints}.
\end{example}

\subsection{Homology decomposition for the Milnor fiber}

The Milnor fiber $M_{f|(X,0)}$ of a holomorphic function $f$ on 
a complex analytic space $(X,0) \subset (\CC^n,0)$ is 
by construction a topologically stable object: 
By virtue of Thom's Isotopy Lemma,
small perturbations 
of the defining equation $f$ do not alter $M_{f|(X,0)}$ up 
to homeomorphism. 
Consequently, in a morsification $F = (f_t,t)$ of $f|(X,0)$
we may identify the Milnor fiber $M_{f|(X,0)}$ 
\[
  M_{f|(X,0)} = B_\varepsilon \cap X \cap f^{-1}(\{\delta\}) \cong 
  B_\varepsilon \cap X \cap f_\eta^{-1}(\{\delta\})
\]
and the generic fiber 
$B_\varepsilon \cap X \cap f_\eta^{-1}(\{\delta\})$ of $f_\eta$ 
for suitable choices of $\varepsilon \gg \delta \gg \eta > 0$. 
For the previous example this is illustrated in the first two pictures 
of Figure \ref{fig:MorsificationProcessI}.

\begin{figure}[h]
  \centering
  \includegraphics[scale=0.14]{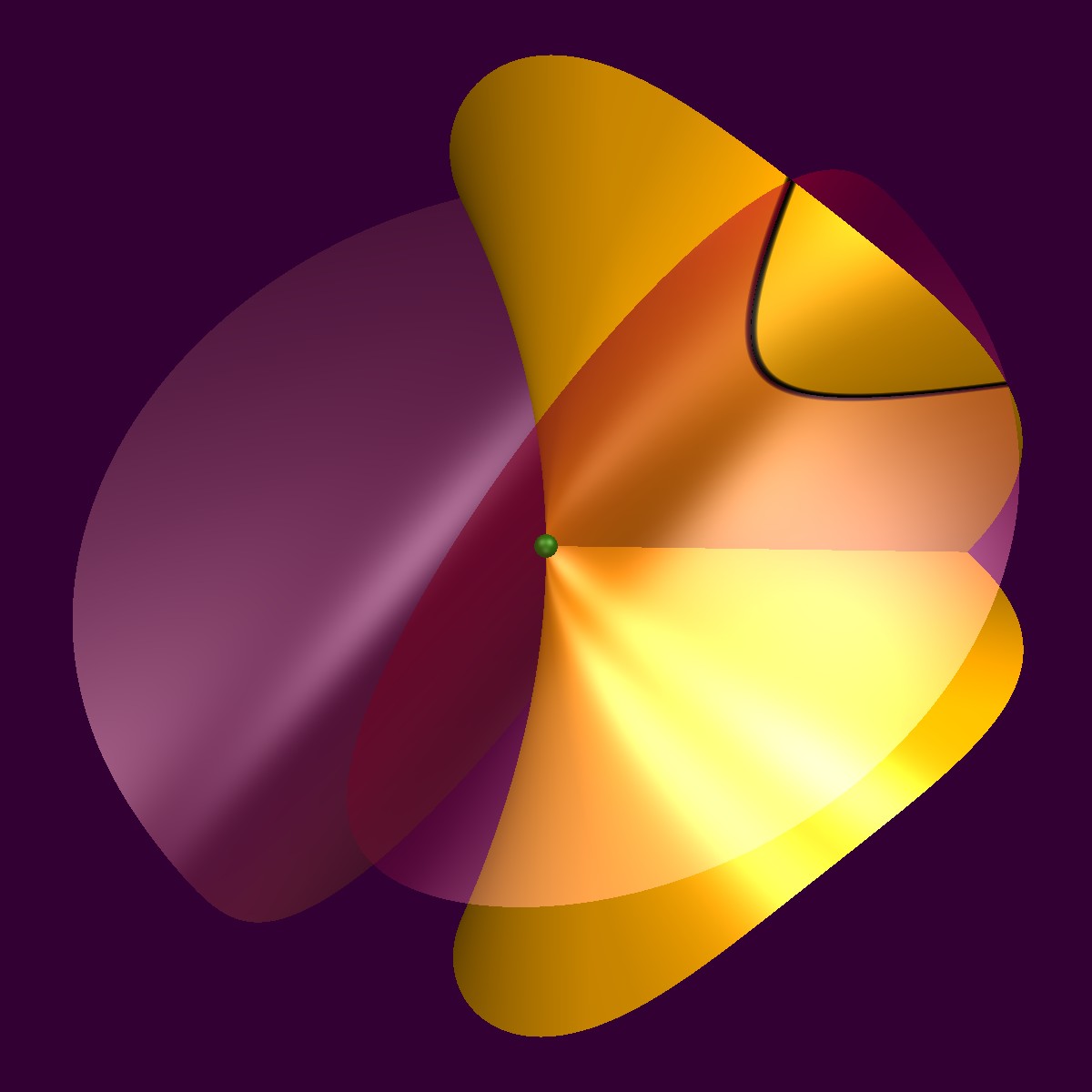}
  \includegraphics[scale=0.14]{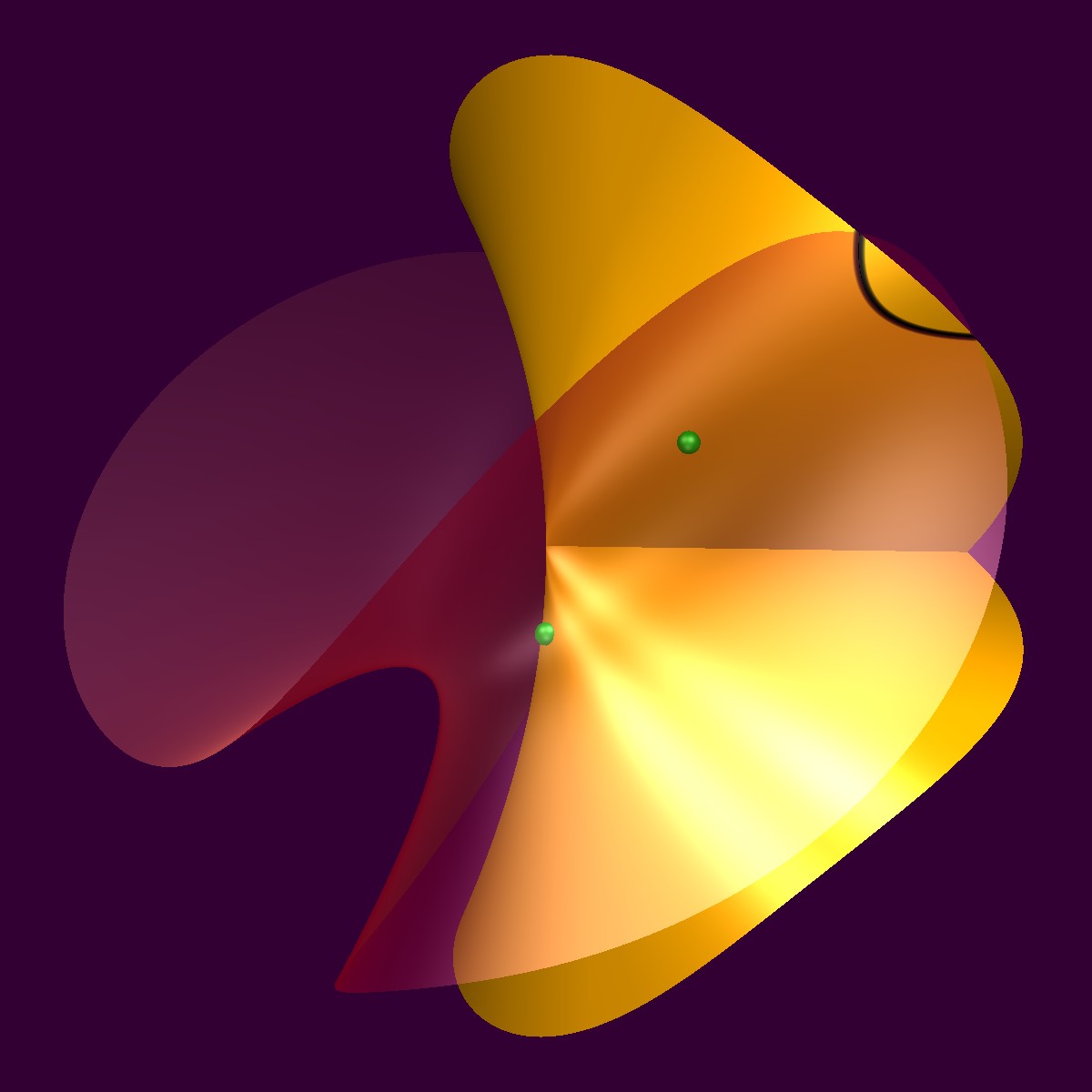}\\
  \vspace{1mm}
  \includegraphics[scale=0.14]{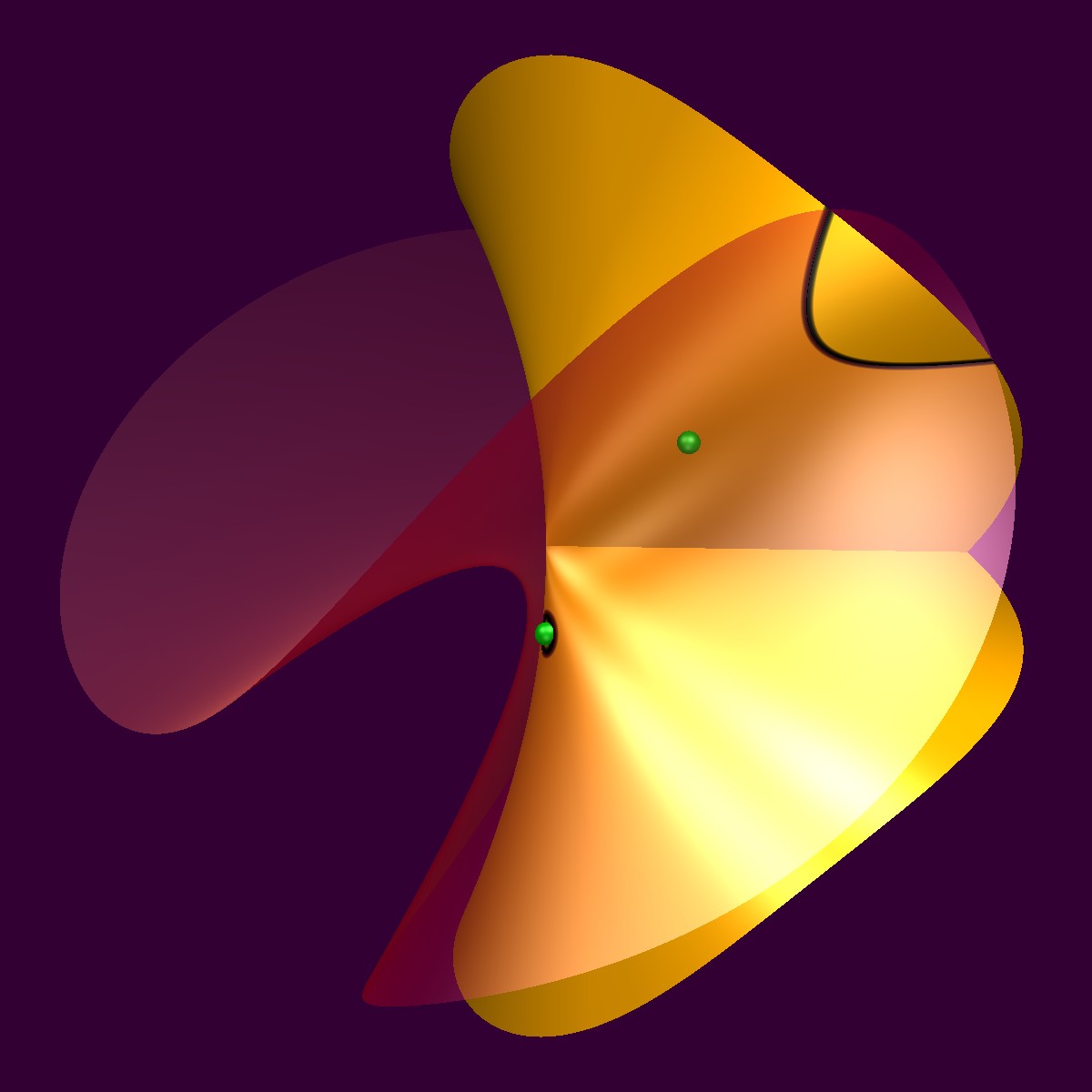}
  \includegraphics[scale=0.14]{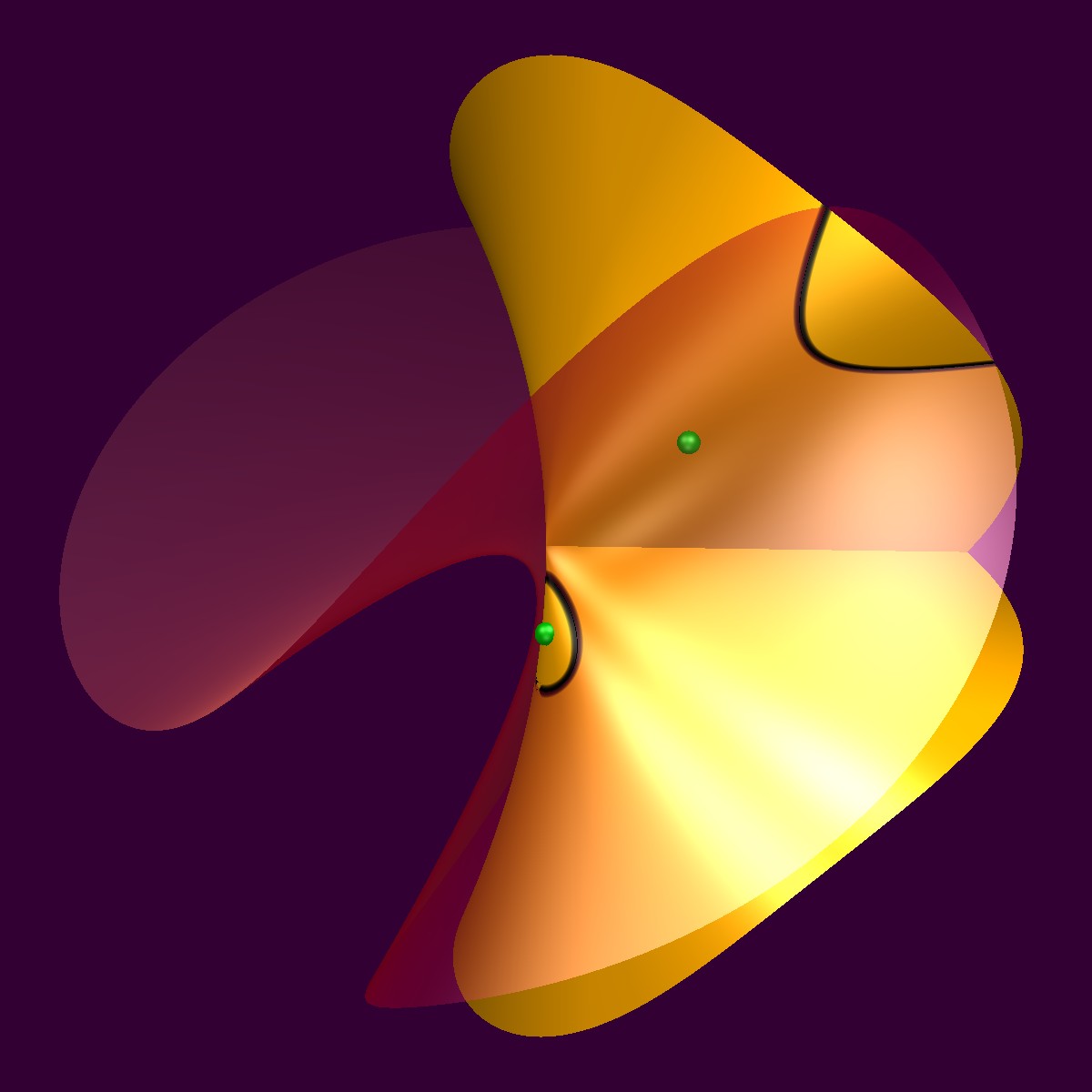}\\
  \caption{Morsification of $f|X$ in a Milnor ball: 1) the Milnor fiber of $f|X$, 
  2) the same fiber of $f_\eta|X$. 3) and 4) depict passing the first critical value of $f_\eta|X$. }
  \label{fig:MorsificationProcessI}
\end{figure}

It is a straightforward exercise to transfer the classical theory of 
morsifications (see e.g. \cite{ArnoldGuseinZadeVarchenkoVolII}) 
to this setting and use stratified Morse theory \cite[Part II]{GoreskyMacPherson88}
to deduce the following homology decomposition 
for the Milnor fiber:

\begin{proposition}
  \label{prp:HomologyDecompositionStratifiedMorsification}
  Let $(X,0) \subset (\CC^n,0)$ be a complex analytic space,
  $S = \{\mathscr S_\alpha\}_{\alpha \in A}$ a complex analytic Whitney 
  stratification of $X$ with connected strata, $\mathcal L(X,\mathscr S_\alpha)$
  the complex link of $X$ along the stratum $\mathscr S_\alpha$, 
  $C(\mathcal L(X, \mathscr S_\alpha))$ the real cone over it, 
  \[
    f \colon (\CC^n,0) \to (\CC,0)
  \]
  a holomorphic function with an isolated singularity on $(X,0)$ 
  in the stratified sense, and $M_{f|(X,0)}$ its Milnor fiber on $X$. 
  Then the reduced homology of the Milnor fiber decomposes as 
  \begin{equation}
    \tilde H_\bullet\left(M_{f|(X,0)} \right) \cong 
    \bigoplus_{\alpha \in A} \bigoplus_{i=1}^{\mu_f(\alpha;X,0)} 
    H_{\bullet-d(\alpha)+1}\left( C(\mathcal L (X,S_\alpha)), \mathcal L(X,S_\alpha)\right),
    \label{eqn:HomologyDecompositionMilnorFiber}
  \end{equation}
  where $d(\alpha) = \dim( \mathscr S_\alpha )$ is the complex dimension 
  of the stratum $\mathscr S_\alpha$  and $\mu_f(\alpha;X,0)$ the number of 
  Morse critical points on $\mathscr S_\alpha$ in a morsification of $f$.
\end{proposition}

Proposition \ref{prp:HomologyDecompositionStratifiedMorsification} shows that the 
characterizations 1') and 2') 
of the numbers $\mu_f(\alpha;X,0)$ 
in Section \ref{sec:Summary} 
coincide.
We include a brief proof.

\begin{proof}
  Choose $\varepsilon > 0$ sufficiently small so that the squared distance 
  function to the origin 
  $r^2 : \CC^n \to \RR_{\geq 0}$ 
  does not have any critical points in the ball $B_\varepsilon$ neither 
  on $X$ nor on $X \cap f^{-1}(\{0\})$. After shrinking $\varepsilon>0$ 
  once more, if necessary, we may assume that the space 
  $B_\varepsilon \cap X \cap f^{-1}(\{0\})$ is a deformation retract of 
  $B_\varepsilon \cap X \cap f^{-1}(D_\delta) $
  for sufficiently small 
  $\varepsilon \gg \delta > 0$. In particular, the space $B_\varepsilon \cap X \cap 
  f^{-1}(D_\delta)$ is contractible. 

  Its boundary $\partial( B_\varepsilon \cap X \cap f^{-1}(D_\delta))$ is 
  topologically stable under small perturbations of $f$. So is the Milnor fiber 
  \[
    M_{f|(X,0)} = B_\varepsilon \cap X \cap f^{-1}(\{\delta\}) \subset 
    \partial( B_\varepsilon \cap X \cap f^{-1}(D_\delta)).
  \]
  In any unfolding $F = (f_t,t)$ of $f$ we may therefore identify the pairs 
  \[
    \left(B_\varepsilon \cap X \cap f^{-1}(D_\delta), M_{f|(X,0)}\right)
    \cong
    \left( B_\varepsilon \cap X \cap f_{\eta}^{-1}(D_\delta), 
    B_\varepsilon \cap X \cap f_\eta^{-1}(\{\delta\})\right)
  \]
  for sufficiently small $\varepsilon \gg \delta \gg \eta >0$. 
  
  After modifying $f_\eta$ a little more we may assume that all critical
  values $\{c_i\}_{i=1}^N$ of $f_\eta$ are distinct points in the disc $D_\delta$. 
  Choose non-intersecting differentiable paths $\gamma_i\colon [0,1] \to D_\delta$ 
  from $\delta$ to 
  $c_i$ and let $\gamma_i([0,1])$ be its image in $D_\delta$. 
  By virtue of Thom's First Isotopy Lemma, the map 
  \[
    f_\eta \colon B_\varepsilon \cap X \cap f_\eta^{-1}(D_\delta) \to 
    D_\delta
  \]
  is a $C^0$-fiber bundle away from the points $c_i$ and the space
  $B_\varepsilon \cap X \cap f_\eta^{-1}(D_\delta)$ retracts onto 
  $f_\eta^{-1}\left( \bigcup_{i=1}^N \gamma_i([0,1]) \right)$. 

  Along each path $\gamma_i$, one attaches a so called \textit{thimble} 
  to $B_\varepsilon \cap X \cap f_\eta^{-1}(\{\delta\}) \cong M_{f|(X,0)}$. 
  This thimble is given by the product of the tangential 
  and the normal morse datum of $f_\eta$ at the critical point 
  $p_i$ over $c_i$. See \cite{GoreskyMacPherson88} for a definition of 
  these. Altogether, we obtain 
  \begin{eqnarray*}
    & & \tilde H_\bullet( M_{f|(X,0)} ) \\
    &=& H_{\bullet+1}\left( B_\varepsilon \cap X \cap f^{-1}(D_\delta), M_{f|(X,0)} \right) \\
    &=& H_{\bullet+1}\left( B_\varepsilon \cap X \cap f_\eta^{-1}(D_\delta), 
    B_\varepsilon \cap X \cap f_\eta^{-1}(\{\delta\}) \right) \\
    &=& H_{\bullet+1}\left( B_\varepsilon \cap X \cap f_\eta^{-1}
    \left( \bigcup_{i=1}^N \gamma_i([0,1]) \right), 
    B_\varepsilon \cap X \cap f_\eta^{-1}(\{\delta\}) \right) \\
    &=& H_{\bullet+1}\left( B_\varepsilon \cap X \cap f^{-1}_\eta
    \left( \bigcup_{i=1}^N \gamma_i([0,1]) \right), 
    B_\varepsilon \cap X \cap f_\eta^{-1}
    \left( \bigcup_{i=1}^N \gamma_i([0,1)) \right) \right) \\
    &=& \bigoplus_{\alpha \in A} \bigoplus_{i=1}^{\mu_f(\alpha;X,0)} 
    H_{\bullet+1}\left( ([0,1], \partial [0,1])^{d(\alpha)} \times 
    ( C(\mathcal L (X,\mathscr S_\alpha)), \mathcal L(X,\mathscr S_\alpha))\right)\\
    &=& \bigoplus_{\alpha \in A} \bigoplus_{i=1}^{\mu_f(\alpha;X,0)} 
    H_{\bullet-d(\alpha)+1}\left( C(\mathcal L (X,\mathscr S_\alpha)), 
    \mathcal L(X,\mathscr S_\alpha)\right).
  \end{eqnarray*}
\end{proof}

\begin{remark}
  The existence of the homology decomposition (\ref{eqn:HomologyDecompositionMilnorFiber})
  also follows from the more general bouquet decomposition of the Milnor 
  fiber due to M. Tib{\u a}r \cite{Tibar95}. His proof, however, does not 
  use morsifications and it requires further work to show that the numbers which 
  play the corresponding role of the $\mu_f(\alpha;X,0)$ in his homology decomposition 
  coincide with the number of Morse critical points in a morsification.
\end{remark}
\begin{figure}[h]
  \includegraphics[scale=0.14]{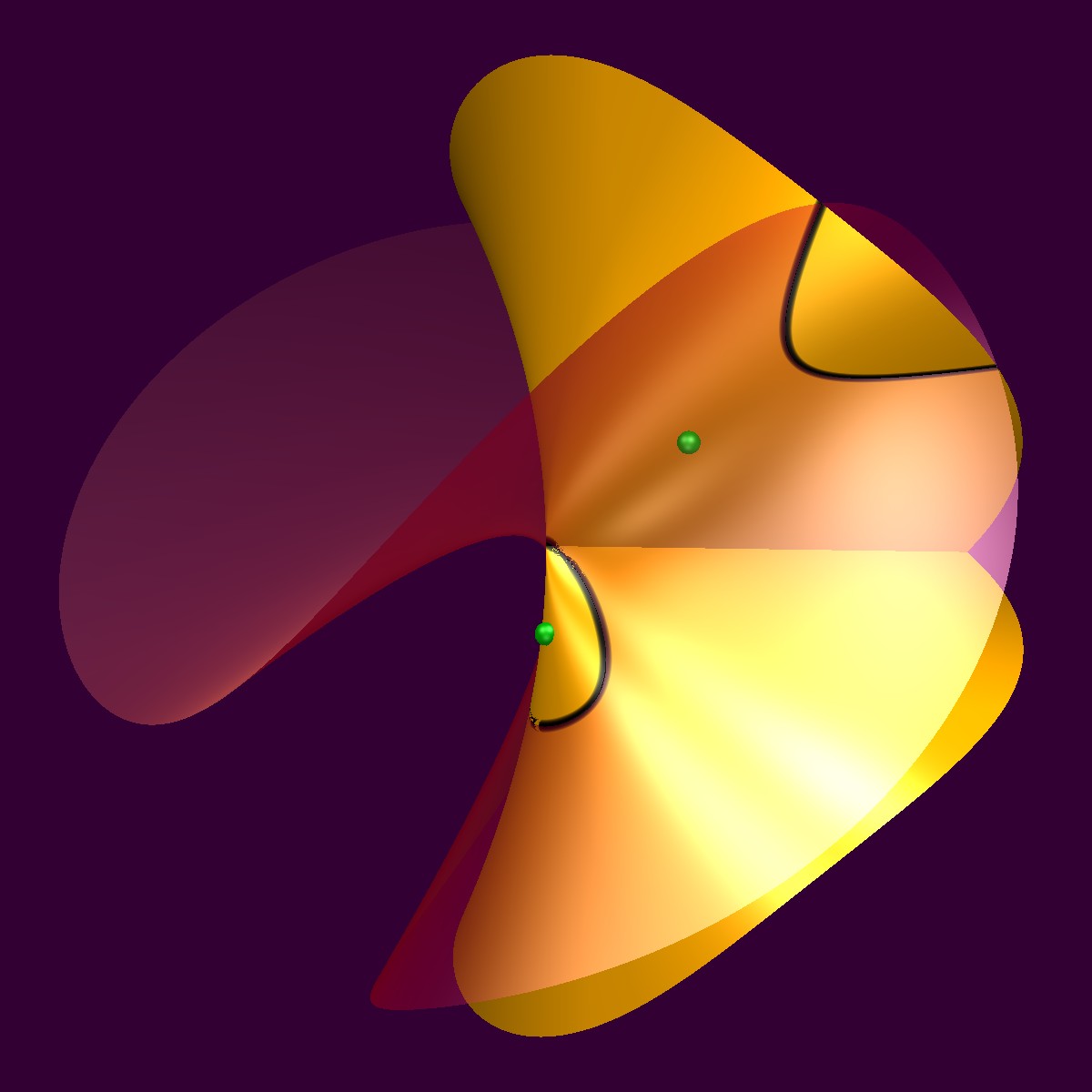}
  \includegraphics[scale=0.14]{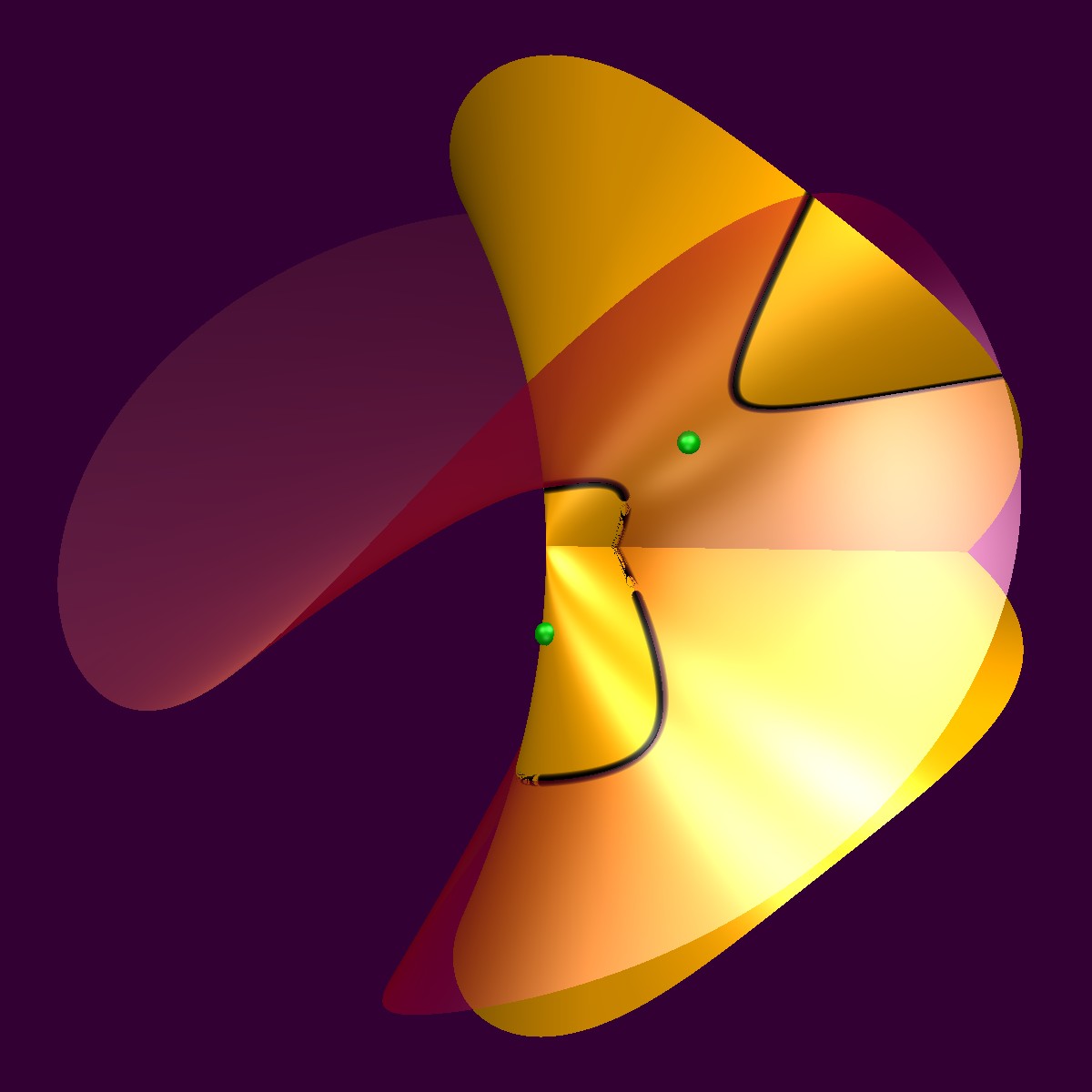}\\
  \vspace{1mm}
  \includegraphics[scale=0.14]{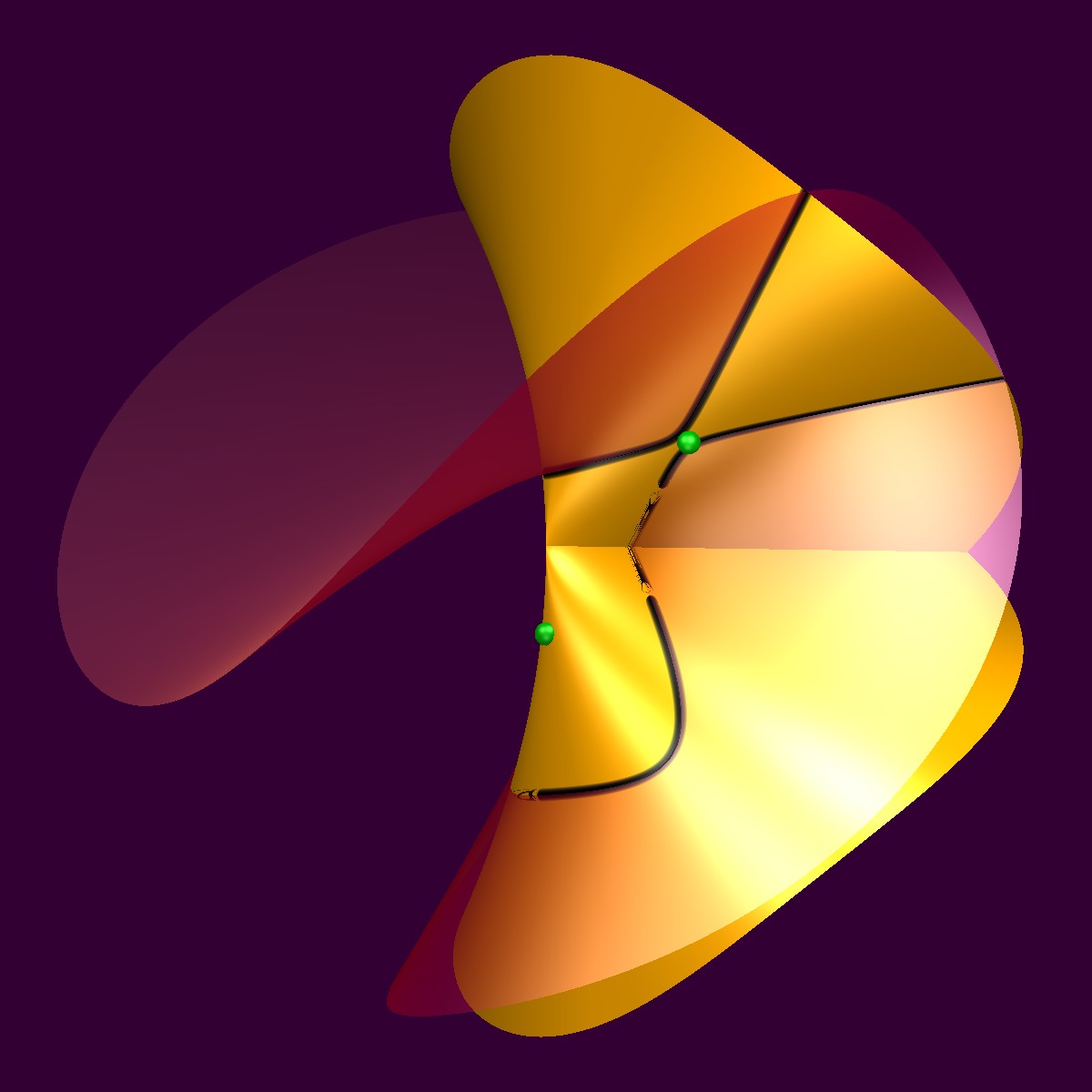}
  \includegraphics[scale=0.14]{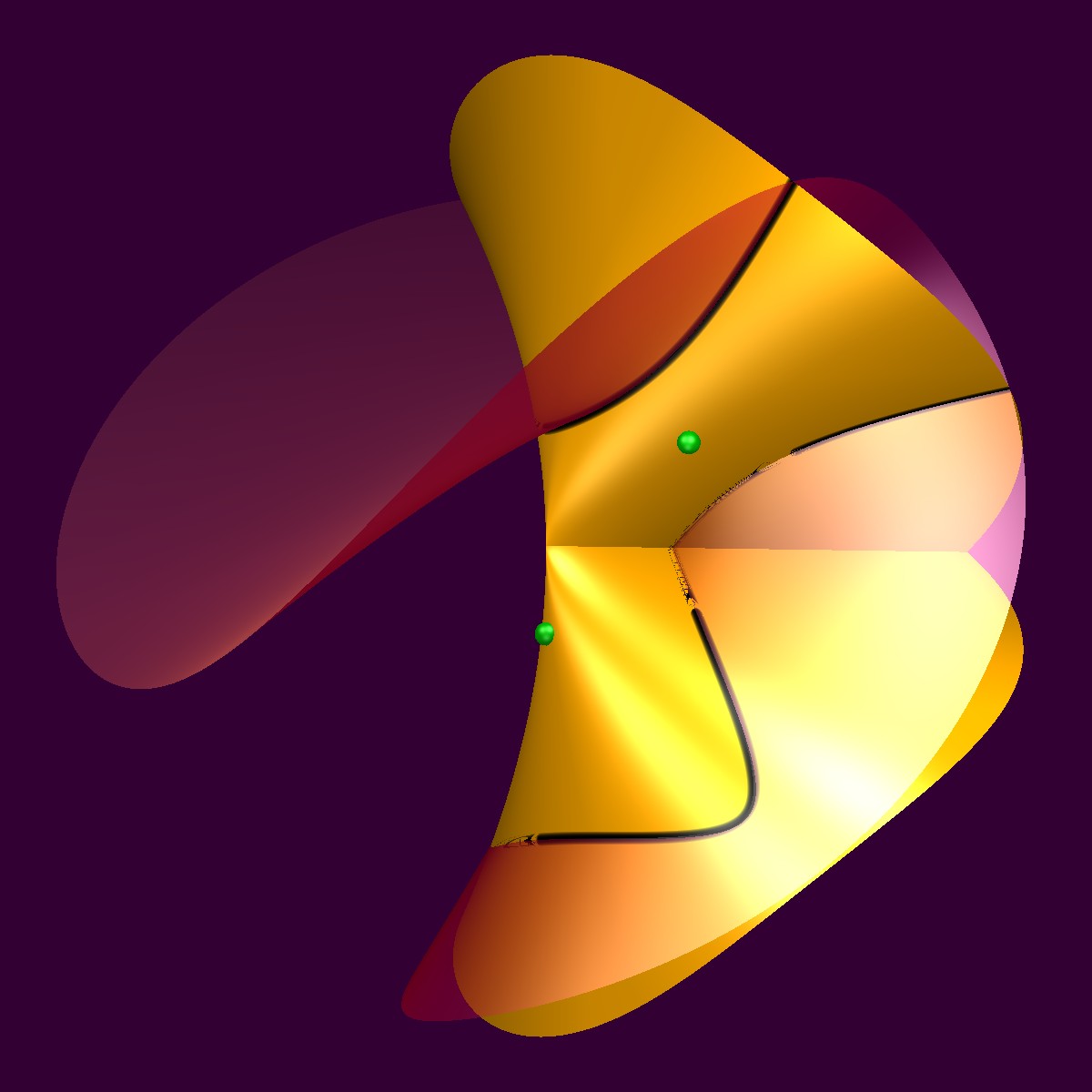}\\
  \vspace{1mm}
  \caption{Morsification of $f|X$ in a Milnor ball: 1) and 2) passing the 
    critical point of $f_\eta|X$ on $\mathscr S_0$, 3) and 4) passing another two critical points 
    on $\mathscr S_2$, one being on the backside.
  }
  \label{fig:MorsificationProcessII}
\end{figure}

\begin{example}
  \label{exp:HomologyDecomposition}
  We continue with Example \ref{exp:Morsification}. For $t=1$ the critical 
  point of the morsified function $f_{1}$ on $\mathscr S_1$ is $(-1/2,0,0)^T$.
  On $\mathscr S_2 \subset X$ they are 
  \[
    p_1 = 
    \begin{pmatrix}
      0 \\ 0 \\ -1
    \end{pmatrix}, \quad 
    p_{2,3} = 
    \begin{pmatrix}
      1\\
      \pm 1 \\
      1
    \end{pmatrix}, \quad 
    p_{4,5} =
    \begin{pmatrix}
      -1 \\
      \pm i\\
      -1
    \end{pmatrix}, \quad 
    p_{6,7} =
    \begin{pmatrix}
      1 \\ \pm 3 \\ -3
    \end{pmatrix}.
  \]
  The complex links $\mathcal L ( X, \mathscr S_\alpha)$ 
  of $X$ 
  along the different strata are the following. 

  For $\mathscr S_0 = \{0\}$, it is the complex link of the Whitney 
  umbrella $(X,0)$ itself, which is known to be the nodal cubic. Hence 
  $\mathcal(X,\mathscr S_0) \cong_h S^1$ is homotopy equivalent to a circle. 
  
  Along $\mathscr S_1$ the normal slice of $X$ consists of two 
  complex lines meeting transversally. The complex link is therefore 
  a pair of points $\mathcal L(X,\mathscr S_1) \cong \{ q_1, q_2 \}$.

  For the third stratum $\mathscr S_2$, the normal slice is a single 
  point and the complex link is empty. We adapt the convention 
  that the real cone over the empty set $C(\emptyset) = \{pt\}$ is the 
  vertex $pt$ of the cone.

  The homology decomposition for the Milnor fiber thus reads 
  \begin{eqnarray*}
    \tilde H_{\bullet}(M_{f|(X,0)}) &\cong&
    H_{\bullet+1}( C(S^1),S^1 ) \oplus 
    H_{\bullet}( C(\{q_1,q_2\}), \{q_1,q_2\}) \oplus 
    \bigoplus_{i=1}^5 H_{\bullet - 1}( \{pt\} ) \\
    &\cong & \ZZ[1] \oplus \ZZ[1] \oplus (\ZZ[1])^5
  \end{eqnarray*}
  where we write $\ZZ[e]$ for a cohomological shift of $\ZZ$ by $e$.
  In combination with the bouquet decomposition theorem from 
  \cite{Tibar95}, we may even infer that $M_{f|(X,0)}$ is homotopy 
  equivalent to a bouquet of seven circles.
\end{example}

\subsection{The Euler obstruction of a $1$-form}
\label{sec:EulerObstruction}

In \cite[Proposition 2.3]{SeadeTibarVerjovsky05}, J. Seade, M. Tib{\u a}r, and 
A. Verjovsky proved that 
\[
  \mu_f(\gamma;X,0 ) = (-1)^{\dim \mathscr S_\gamma} \Eu_f(X,0)
\]
for the top dimensional 
stratum\footnote{Or, in case $(X,0)$ is reducible, the union of the top dimensional strata}
$\mathscr S_\gamma$.
The Euler obstruction of a function 
is defined using the gradient vector field $\grad f$.
For the purposes of this note, it is more natural to consider  
the $1$-form $\D f$ and its canonical lift to the dual $\tilde \Omega^1$ 
of the Nash bundle as we will describe below. 
This provides the notion of the Euler obstruction $\Eu^{\D f}(X,0)$ 
of the $1$-form $\D f$ 
on $(X,0)$, as was 
first defined by W. Ebeling and S.M. Gusein-Zade in \cite{EbelingGuseinZade05}.
In this section, we will follow their example and also consider the 
slightly more general case of an arbitrary $1$-form $\omega$ on $(X,0)$. 

\medskip

Throughout this section, we let 
$U \subset \CC^n$ be an open domain and $X \subset U$ a reduced, complex 
analytic space.
Suppose that $X$ is equidimensional of dimension $d$. 
On the set of nonsingular points 
$X_{\reg}$ we can consider the map 
\begin{equation}
  \Phi \colon X_{\reg} \to \Grass(d,n), \quad p \mapsto \left[ T_p X \subset T_p \CC^n \right]
  \label{eqn:RationalMapForNashBlowup}
\end{equation}
taking any point $p$ to the class of its tangent space $T_p X$ as a subspace of  
$T_p \CC^n$ by means of the embedding of $X$.

\begin{definition}
  The Nash modification of $X$ is the complex analytic closure of the graph
  \[
    \tilde X = \overline{\{ (p,\Phi(p)) : p \in X_{\reg}\} } \subset U \times \Grass(d,n)
  \]
  together with its projections 
  \[
    \xymatrix{ 
      & \tilde X \ar[dl]_{\nu} \ar[dr]^{\rho} & \\
      X & & \Grass(d,n).
    }
  \]
  The restriction of the tautological bundle on $U \times \Grass(d,n)$ 
  to $\tilde X$ will be referred to as the Nash bundle $\tilde T$. 
  The dual bundle will be denoted by $\tilde \Omega^1$. 
\end{definition}

For the dual of the Nash bundle there is a natural notion of 
pullback of $1$-forms on $X$ which is defined as follows. 
We can think of a point $(p,V) \in \tilde X$ as a pair 
of a point $p\in X$ and a limiting tangent space $V$ 
from $X_{\reg}$ at $p$.
The space $V$ can be considered both as a 
subspace of $T_p \CC^n$ and 
as the fiber of the Nash bundle $\tilde T$ 
at the point $(p,V)$. 
Let us denote by $\langle \cdot, \cdot \rangle$ 
the canonical pairing between a vector space and its dual.
For a $1$-form $\omega$ on $\CC^n$, a limiting tangent 
space $V$ at 
$p$ and a vector $v \in V$ we define 
\begin{equation}
  \langle \nu^* \omega(p,V), v \rangle := \langle \omega(p), v \rangle.
  \label{eqn:PullbackOfOneForms}
\end{equation}
Here we consider $v$ as a point in the fiber of the Nash bundle 
over the point $(p,V) \in \tilde X$ on the left hand side and as a vector in $V \subset T_p \CC^n$ 
on the right hand side.

\medskip
In order to define the Euler obstruction of a $1$-form, we need 
to adapt Definitions \ref{def:StratifiedRegularPoint} and \ref{def:StratifiedMorseCriticalPoint}
in this setup. Since for $1$-forms there is no associated Milnor fibration, 
we may drop the assumption that the stratification of $X$ satisfies Whitney's 
condition B. 

Let $\omega$ be a holomorphic $1$-form on $U$. 

\begin{definition}
  \label{def:StratifiedNonZeroPointOfAOneForm}
  Suppose $S = \{ \mathscr S_\alpha \}_{\alpha \in A}$ is a complex analytic 
  stratification of $X$ satisfying Whitney's condition A.

  We say that $\omega|(X,p)$ is nonzero at a point $p\in X$ in the stratified 
  sense if $\omega$ does not vanish on the tangent space $T_p \mathscr S_\beta$ 
  of the stratum $\mathscr S_\beta$ containing $p$. 

  We say that a $1$-form $\omega$ on $U$ has an isolated zero on $(X,p)$, 
  if there exists an open neighborhood 
  $U'$ of $p$ such that $\omega$ is nonzero on $X$ in the stratified sense at 
  every point $x \in U' \cap X\setminus \{p\}$.
\end{definition}

\noindent If in the following we do not specify a stratification, 
we again choose $S$ to be the canonical Whitney stratification 
for a reduced, equidimensional complex analytic space $X$.

It is an immediate consequence 
of the Whitney's condition A that at every point $p\in X$ such 
that the restriction $\omega|\mathscr S_\alpha$ 
of $\omega$ to the stratum $\mathscr S_\alpha$ containing $p$ 
is non-zero, also the pullback $\nu^* \omega$ is non-zero 
at \textit{any} point $(p,V) \in \nu^{-1}(\{p\})$ in the fiber 
of $\nu \colon \tilde X \to X$ over $p$. In particular, 
$\nu^* \omega$ is a nowhere vanishing section on the preimage 
of a punctured neighborhood $U'$ of $p$ whenever $\omega$ 
has an isolated zero on $(X,p)$ in the stratified sense. 

\begin{definition}[cf.\,\cite{EbelingGuseinZade05}]
  \label{def:EulerObstructionOfAOneForm}
  Let $(X,p) \subset (\CC^n,p)$ be an equidimensional, reduced, 
  complex analytic space of dimension $d$ 
  and $\omega$ the germ of a $1$-form 
  on $(\CC^n,p)$ such that $\omega|(X,p)$ has an isolated zero 
  in the stratified sense. The Euler obstruction $\Eu^\omega(X,p)$
  of $\omega$ 
  on $(X,p)$ is defined as the obstruction to extending 
  $\nu^* \omega$ as a nowhere vanishing section of the dual 
  of the Nash bundle from the preimage $\nu^{-1}( \partial B_\varepsilon \cap X )$ 
  of the real link $\partial B_\varepsilon \cap X$ of $(X,p)$ 
  to the interior of $\nu^{-1}( B_\varepsilon \cap X)$ of 
  the Nash transform. More precisely, it is the value 
  of the obstruction class 
  \[
    \Obs(\nu^* \omega) \in 
    H^{2d}\left( \nu^{-1}(B_\varepsilon \cap X), \nu^{-1}(\partial B_\varepsilon \cap X )\right)
  \]
  of the section $\nu^*\omega$ on the fundamental class of the pair 
  $\left( \nu^{-1}(B_\varepsilon \cap X), \nu^{-1}(\partial B_\varepsilon \cap X )\right)$:
  \[
    \Eu^{\omega}(X,p) = 
    \left\langle \Obs(\nu^* \omega), 
    \left[ \nu^{-1}( B_\varepsilon \cap X), \nu^{-1}(\partial B_\varepsilon \cap X) \right]
    \right\rangle.
  \]
\end{definition}

As we shall see below, the Euler obstruction of a $1$-form $\omega$ 
with isolated singularity on $(X,p)$ counts the zeroes on 
$X_{\reg}$ of a generic deformation $\omega_\eta$ of $\omega$. 
In the case $\omega = \D f$ for some function $f$ with isolated 
singularity on $(X,p)$, these zeroes correspond to Morse critical 
points of $f_\eta$ on $X_{\reg}$ in an unfolding. 
We have seen before that these are not the only critical points 
of $f_\eta$. 
\begin{definition}
  \label{def:SimpleStratifiedZeroOfAOneForm}
  Suppose $S = \{ \mathscr S_\alpha\}_{\alpha\in A}$ is a complex analytic 
  stratification of $X$ satisfying Whitney's condition A.
  A point $p \in X$ is a \textit{simple zero} of $\omega|X$, if the following holds. 
  Let $\mathscr S_\beta$ be the stratum containing $p$ and 
  $\sigma( \omega|\mathscr S_\beta)$ the section of the restriction 
  $\omega|\mathscr S_\beta$ as a submanifold of the total space of the vector bundle 
  $\Omega^1_{\mathscr S_\beta}$.
  Denote the zero section by $\sigma(0)$.

  \begin{enumerate}[i)]
    \item The intersection of $\sigma(\omega|\mathscr S_\beta)$ and the zero section
      \[
	\sigma( \omega|\mathscr S_\beta) \pitchfork_p \sigma(0)
      \]
      in the vector bundle $\Omega^1_{\mathscr S_\beta}$ on $\mathscr S_\beta$ 
      is transverse at $p$.
    \item $\omega$ does not annihilate any limiting tangent space $V$ 
      from a higher dimensional stratum at $p$.
  \end{enumerate}
\end{definition}

Whenever $\omega = \D f$ for some holomorphic function $f$, this 
reduces precisely to the definition of a stratified Morse critical point 
$p$ of $f|X$, Definition \ref{def:StratifiedMorseCriticalPoint}.

\medskip
Analogous to morsifications we define an unfolding of a $1$-form 
$\omega$. Since $\Omega^1_{U}$ is trivial, we can 
consider $\omega$ as 
a holomorphic map $U \to \CC^n$. An unfolding of $\omega$ 
is then given by a holomorphic map germ 
\[
  W \colon \left( \CC^n \times \CC, (p,0) \right) \to 
  \left( \CC^n \times \CC, (\omega(p),0) \right),
  (x,t) \mapsto (\omega_t(x), t).
\]

\begin{proposition}
  \label{prp:MorsificationLemmaForOneForms}
  Any $1$-form $\omega$ with an isolated zero on $(X,p)$ admits an unfolding 
  $W = (\omega_t,t)$ as above on some open sets $U' \times T$
  such that for a sufficiently small ball $B_\varepsilon \subset U'$ around $p$ 
  and an open subset $0 \in T' \subset T$ one has 
  \begin{enumerate}[i)]
    \item $X\cap B_\varepsilon$ retracts onto the point $p$,
    \item $\omega = \omega_0$ on $U'$ and $\omega$ has an isolated zero on 
      $X \cap U'$,
    \item for every $t \in T'$, $t \neq 0$, the $1$-form 
      $\omega_t$ has only simple isolated zeroes on $X \cap B_\varepsilon$ 
      and is nonzero on $X \cap U'$ at all boundary points $x \in X \cap \partial B_\varepsilon$.
  \end{enumerate}
  Moreover, $\omega_t$ can be chosen to be of the form 
  $\omega_t = \omega - t \cdot \D l$
  for a linear form $l \in \Hom(\CC^n,\CC)$.
\end{proposition}

\begin{definition}
  \label{def:MultiplicityOfOneForms}
  We define the \textit{multiplicity} $\mu^\omega( \alpha; X,p)$ 
  of $\omega|(X,p)$ to be 
  the number of simple zeroes of $\omega_t$ on $\mathscr S_\alpha$ 
  for $t \neq 0$
  in an unfolding as in Proposition \ref{prp:MorsificationLemmaForOneForms}.
\end{definition}

Again, we clearly have $\mu_f(\alpha;X,p) = \mu^{\D f}(\alpha;X,p)$ in the 
case where $\omega = \D f$ is the differential of a function $f$ with 
isolated singularity on $(X,p)$. As a straightforward consequence
we obtain:

\begin{corollary}
  \label{cor:ExistenceOfMorsifications}
  For a holomorphic function $f \colon U \to \CC$ with an isolated singularity 
  in the stratified sense at $(X,p)$ a morsification $F = (f_t,t)$ of $f|(X,p)$ 
  can be chosen to be of the form 
  \[
    f_t = f - t \cdot l
  \]
  for a linear form $l \in \Hom(\CC^n,\CC)$.
\end{corollary}

\begin{proof}
  (of Proposition \ref{prp:MorsificationLemmaForOneForms})
  We will show using Bertini-Sard-type methods that there exists 
  a dense set $\Lambda \subset \Hom(\CC^n,\CC)$ of admissable lines 
  such that the linear form $l$ in Proposition \ref{prp:MorsificationLemmaForOneForms} 
  can be chosen to be an arbitrary linear form with $[l] \in \Lambda$.

  For a fixed $\alpha$ let 
  $X_\alpha= \overline{\mathscr S_\alpha}$ be the closure of the stratum $\mathscr S_\alpha$, 
  $d(\alpha)$ its dimension,
  and $\nu \colon \tilde X_\alpha \to X_\alpha$ its
  Nash transform. 
  Denote the fiber of $\nu$ over the point $p \in X$ by $E$. Since the 
  question is local in $p$, we may restrict our attention to arbitrary small 
  open neighborhoods of $E$ of the form $\nu^{-1}(U')$ for some open set $U' \ni p$.
  Set 
  \[
    N = \left\{ (x,V,\varphi) \in \tilde X_\alpha \times \Hom(\CC^n,\CC) : 
    \varphi| V = \nu^* \omega(x,V) \right\} 
  \]
  and let 
  $\pi \colon N \to \tilde X_\alpha$ and $\rho \colon N \to \Hom(\CC^n,\CC)$ 
  be the two canonical projections. 
  It is easy to see that $N$ has the structure of a principle $\CC^{n-d(\alpha)}$-bundle 
  over $\tilde X_\alpha$. In particular, the open subset 
  ${\mathscr S'}_\alpha = (\nu\circ\pi)^{-1}(\mathscr S_\alpha) \subset N$ 
  is a complex manifold of dimension $n$.

  Let $\Phi \colon N \dashrightarrow \PP(\Hom(\CC^n,\CC))$ be the rational map sending 
  a point $(\varphi,x,V)$ to the class $[\varphi] \in \PP(\Hom(\CC^n,\CC))$. 
  Since $\omega$ had an isolated zero on $(X,p)$, this map is 
  regular on the dense open subset $N\setminus (\pi\circ \nu)^{-1}(\{p\})$ which  
  in particular contains $\mathscr S'_\alpha$.
  In order to work with regular and proper maps, we may resolve the indeterminacy 
  of $\Phi$ and obtain a commutative diagram
  \[
    \xymatrix{
      \hat{\mathscr S}_\alpha \ar@{^{(}->}[r] \ar[d]_\cong& 
      \hat N \ar[drr]^{\hat \Phi} \ar[d] & \\
      {\mathscr S'_\alpha} \ar@{^{(}->}[r] \ar[d] & 
      N \ar[r]^\rho \ar[d]^{\pi} & \Hom(\CC^n,\CC) \ar@{-->}[r] & \PP(\Hom(\CC^n,\CC)). \\
      \tilde{\mathscr S}_\alpha \ar[d]_\cong \ar@{^{(}->}[r] & \tilde X_\alpha \ar[d]^\nu&  & \\
      \mathscr S_\alpha \ar@{^{(}->}[r] & X_\alpha &  & \\
    }
  \]
  Suppose $L \in \PP(\Hom(\CC^n,\CC))$ is a regular value 
  of $\hat\Phi|\hat{\mathscr S}_\alpha$, 
  then $\hat \Phi^{-1}(\{L\}) \cap \hat{\mathscr S}_\alpha$ 
  is a smooth complex analytic curve.
  If we let $C \subset N$ be the image in $N$ of its analytic closure in $\hat N$,
  then evidently $\rho|C \colon C \to L$ is a finite, 
  branched covering at $0 \in L$. 
  It follows a posteriori 
  from the Curve Selection Lemma that $\rho$ is a submersion 
  at every point 
  $(x,V,\varphi) \in C \cap \mathscr S'_\alpha$
  in a neighborhood of $E$.
  An inspection of the differential of $\rho$ at such a point $(x,V,\varphi)$ 
  reveals that 
  the transversality 
  requirement i) in Definition \ref{def:SimpleStratifiedZeroOfAOneForm}
  is satisfied for the $1$-form $\omega - \D \varphi$ at $x$.
  Conversely, this means that for every nonzero linear form $l \in L$ 
  and every sufficiently small $t \neq 0$ 
  the $1$-form $\omega - t \cdot \D l$ has only isolated 
  zeroes at those points $x \in \mathscr S_\alpha$, 
  for which $(t\cdot l, x, V) \in C$. 
  Repeating this process for every stratum, we obtain a dense set 
  $\Lambda_1 \subset \PP(\Hom(\CC^n,\CC))$ of pre-admissable lines.

  In order to verify also the requirement ii) in Definition 
  \ref{def:SimpleStratifiedZeroOfAOneForm}, we proceed as follows.
  Let $Y_\alpha = X_\alpha \setminus \mathscr S_\alpha$ be the union 
  of limiting strata of $\mathscr S_\alpha$ and $\tilde{\mathscr S_\alpha}$, 
  $\mathscr S'_\alpha$, and $\hat{\mathscr S_\alpha}$ their preimages in 
  $\tilde X_\alpha$, $N$, and $\hat N$, respectively.
  The latter three spaces might have rather difficult geometry, but evidently 
  $\dim \hat Y_\alpha < \dim \hat N = n$ and the map 
  $\hat Y_\alpha \to Y_\alpha$ is surjective.
  
  There exists a dense subset $\Lambda_2 \subset \PP(\Hom(\CC^n,\CC))$ 
  such that the restriction $\hat \Phi|\hat Y_\alpha$ has at most 
  discrete fibers over $\Lambda_2$. To see this, we may for example 
  stratify $\hat Y_\alpha$ by finitely many locally 
  closed complex submanifolds $M_i$ and choose $\Lambda_2$ as the 
  set of all regular values of $\hat \Phi| M_i$. Since 
  $\dim M_i \leq \dim \hat Y_\alpha < n$, the fiber $\hat Q = (\hat \Phi|\hat Y_\alpha)^{-1}(L)$ 
  of a point $L \in \Lambda_2$ is discrete and so is its image $Q \subset N$, because 
  $\hat N \to N$ is proper. This means that for a given $l \in L$ there are only 
  finitely many preimages $(x,V,l) \in \rho^{-1}(L)$, i.e. the set of points 
  $x \in X$, for which $\omega - \D l$ annihilates a limiting tangent 
  space $V$ at $x$ is finite in a neighborhood of $p$. 
  We may choose $U'$ and $B_\varepsilon$ sufficiently small to avoid those 
  points.

  We conclude the proof by setting $\Lambda = \Lambda_1 \cap \Lambda_2$.
\end{proof}

We are now prepared to show equivalence of 1') and 3'), 
cf. \cite[Proposition 2.3]{SeadeTibarVerjovsky05}.
\begin{proposition}
  \label{prp:MuEqualsEu}
  For every $1$-form $\omega$ on $U$ with an isolated zero 
  on $(X,p)$ we have 
  \[
    \mu^\omega(\alpha; X,p) = \Eu^{\omega}(X_\alpha,p),
  \]
  where $X_\alpha = \overline{\mathscr S_\alpha}$ is the closure 
  of the stratum $\mathscr S_\alpha$.
\end{proposition}

\begin{proof}
  Choose a representative 
  \[
    W = (\omega_t,t) \colon U' \times T \to \CC^n \times T
  \]
  of an unfolding of $\omega|(X,p)$ and a ball $B_\varepsilon \subset U'$ as in 
  Proposition \ref{prp:MorsificationLemmaForOneForms}. The 
  Euler obstruction of $\omega$ at $(X_\alpha,p)$ depends only on 
  its obstruction class 
  \[
    \Obs(\nu^*\omega) \in 
    H^{2d}(\nu^{-1}(B_\varepsilon \cap X_\alpha), \nu^{-1}(\partial B_\varepsilon \cap X_\alpha)).
  \]
  Being a homotopy invariant, this class does not change under small perturbations and it 
  is therefore evident from the definitions that for every $\eta \in T$ and every $\alpha \in A$ 
  one has 
  \begin{eqnarray*}
    \Eu^\omega(X_\alpha,p) &=& 
    \left\langle \Obs(\nu^* \omega), 
    \left[ \nu^{-1}( B_\varepsilon \cap X_\alpha), \nu^{-1}(\partial B_\varepsilon \cap X_\alpha) \right]
    \right\rangle \\
    &=& 
    \left\langle \Obs(\nu^* \omega_t), 
    \left[ \nu^{-1}( B_\varepsilon \cap X_\alpha), \nu^{-1}(\partial B_\varepsilon \cap X_\alpha) \right] 
    \right\rangle. 
  \end{eqnarray*}
  We may therefore select one $\eta \neq 0$ and use $\omega_\eta$ 
  instead 
  of $\omega$ to compute the Euler obstruction. The evaluation of the 
  obstruction class 
  counts the number of zeroes of $\omega_\eta$. Observe that by construction, $\nu^* \omega$ 
  is nonzero at any 
  point $(x,V) \in \tilde X_\alpha \setminus \nu^{-1}(\mathscr S_\alpha)$, 
  because $\omega_\eta$ does not annihilate any limiting tangent space $V$ at $x$.
  Thus, the zeroes of $\nu^* \omega_\eta$ are located in $\nu^{-1}( \mathscr S_\alpha)$.
  At every such zero $(x,V) \in \nu^{-1}(\mathscr S_\alpha)$ of $\omega_\eta$ 
  the intersection 
  of $\sigma(\omega_\eta|\mathscr S_\alpha)$ and the zero section in 
  $\Omega^1_{\mathscr S_\alpha}$ is transverse with positive orientation 
  and therefore contributes an increment of 
  $1$ to the Euler obstruction. 
  Consequently, $\Eu^\omega(X_\alpha,p)$ coincides with $\mu^\omega(\alpha; X,p)$.
\end{proof}

\begin{corollary}
  \label{cor:MuEqualsEuForFunctions}
  Whenever $f \colon U \to \CC$ is a holomorphic function 
  with isolated singularity on $(X,p)$, we have 
  \[
    \mu_f(\alpha;X,p) = \Eu^{\D f}(X_\alpha,p).
  \]
\end{corollary}

\begin{example}
  \label{exp:EulerObstructionOfAOneForm}
  We continue with Example \ref{exp:HomologyDecomposition}. For 
  $\alpha = 0$ the real link of $(\overline{\mathscr S_0},0)$ 
  is empty and the Euler obstruction is $1$ by convention.

  In the case $\alpha = 1$ the closure $X_1 = \overline{\mathscr S}_1$ of 
  the stratum $\mathscr S_1$ is 
  already a smooth line. Consequently, the Nash modification 
  $\nu \colon \tilde X_1 \to X_1$ is an isomorphism and 
  $\tilde \Omega^1$ coincides with the usual sheaf of 
  K\"ahler differentials. In this case, the Euler obstruction 
  of $\D f$ on $(X_1,0)$ coincides with the degree of 
  the map 
  \[ 
    \frac{\D f}{|\D f|} \colon \partial B_\varepsilon \cap X_1 \to S^1.
  \]
  Since $0 \in X_1$ is a classical Morse critical point, 
  $\D f$ has a simple, isolated zero on $(X_1,0)$ and therefore 
  \[
    \Eu^{\D f}(X_1,0) = \deg \frac{\D f}{|\D f|} = 1.
  \]
  In this particular case of a function on a complex line, the computation 
  of the Euler obstruction reduces to Rouch\'e's theorem.

  For $\alpha = 2$ we really need to work with the Nash modification and 
  the morsification $F=(f_t,t)$ of $f|(X,0)$.
  To this end, we identify $\Grass(2,3)$ with its dual Grassmannian 
  $\Grass(1,3) \cong \PP^2$ via 
  \[
    V \mapsto V^\perp = \{ \varphi \in \Hom(\CC^3,\CC) : \varphi|V = 0 \}.
  \]
  In homogeneous coordinates $(s_0:s_1:s_2)$ of $\PP^2$ the rational 
  map $\Phi$ from (\ref{eqn:RationalMapForNashBlowup}) is given by 
  the differential of $h$:
  \[
    \Phi \colon \mathscr S_2 \to \PP^2, \quad 
    \begin{pmatrix}
      x \\ y\\ z
    \end{pmatrix} 
    \mapsto 
    \begin{pmatrix}
      s_0 \\ s_1 \\ s_2
    \end{pmatrix} =
    \begin{pmatrix}
      -z^2 \\ 2y \\ -2xz
    \end{pmatrix}
  \]
  The equations for $\tilde X \subset \PP^2 \times \CC^3$ are rather 
  complicated, but they simplify in the canonical charts of $\PP^2 \times \CC^3$. 
  We will consider the chart $s_0 \neq 0$, leaving the computations 
  in the other charts to the reader. The equations for $\tilde X$ read 
  \[
    x = \frac{1}{4}z^2s_1^2, \quad 
    y = -\frac{1}{2}z^2s_1, \quad 
    s_2 = \frac{1}{2}zs_1^2.
  \]
  In particular, we can use $(z,s_1)$ as coordinates on $\tilde X \cap \{ s_0 \neq 0 \} \cong \CC^2$.
  The exceptional set $E \subset \tilde X$, 
  i.e. the set of points $q\in \tilde X$, at which $\nu \colon \tilde X \to X$ is not 
  a local isomorphism, is the preimage of the $x$-axis in $\CC^3$. In the above coordinates 
  it is given by 
  \[
    E = \{ z = 0 \} = \{ 0 \} \times \CC \subset \CC^2 \cong \tilde X \cap \{ s_0 \neq 0 \}.
  \]
  
  Let $\OO(-1)$ be the (relative) tautological bundle on $\PP^2 \times \CC^3$.
  The dual bundle $\OO(1)$ has a canonical set of global sections $e_0,e_1,e_2$ 
  in correspondence with the homogeneous coordinates $(s_0:s_1:s_2)$.
  With these choices the differential of $f_t = y^2 - (x-z)^2 - t(x+2z)$ pulls back to 
  \begin{eqnarray*}
    \nu^* \D f_t &=&  (-2(x-z) - t) \cdot  e_0 
    + 2y \cdot  e_1 
    + (2(x-z) - 2 t)\cdot  e_2
  \end{eqnarray*}
  We consider $\nu^* \D f_t$ as a section in $\tilde \Omega^1$, the dual of 
  the Nash bundle $\tilde T$. Note that $\tilde T$
  appears as part of the Euler sequence
  \[
    \xymatrix{
      0 \ar[r]&
      \tilde T \ar[r]&
      \mathcal O^3_{\tilde X} \ar[r]&
      \mathcal \OO_{\tilde X}(1) \ar[r]&
      0
    }
  \]
  on $\tilde X$. The standard trivialization of $\tilde T$ in the chart $s_0 \neq 0$
  is given by the sections 
  \[
    v_1 =
    \begin{pmatrix}
      -s_1 \\ 1 \\
      0
    \end{pmatrix}, \qquad
    v_2 =
    \begin{pmatrix}
      -s_2 \\ 0 \\ 1
    \end{pmatrix}
  \]
  and therefore the zero locus of $\nu^* \D f_t$ on $\tilde X$ is given by 
  the equations $\nu^*\D f_t(v_1) = \nu^* \D f_t(v_2) = 0$. Substituting 
  all the above expressions we obtain 
  \begin{eqnarray*}
    \nu^* \D f_t(v_1) &=&  (-s_1) \left( z^2 - \frac{1}{2}z^2 s_1^2 + 2z -t\right)\\
    \nu^* \D f_t(v_2) &=&  \left( 1 + \frac{1}{2}z s_1^2 \right)\cdot 
    \left( \frac{1}{2} z^2 s_1^2 - 2z \right) + t 
    \left( \frac{1}{2} z s_1^2 - 2 \right).
  \end{eqnarray*}
  It is easy to see that for $t = 0$ the exceptional set $E = \{z=0\}$ 
  is contained in the zero locus of $\nu^* \D f_0$. In particular, the zero locus is 
  non-isolated and we can not use $\nu^* \D f_0$ to compute 
  the Euler obstruction as in the proof of Proposition \ref{prp:MuEqualsEu}. 
  
  For $\eta \neq 0$, however, the zero locus of $\nu^* \D f_\eta$ consists of 
  only finitely many points. 
  A primary decomposition reveals that there are seven branches 
  \[
    \tilde \Gamma_1(t) =
    \begin{pmatrix}
      -t \\ 0 
    \end{pmatrix}, \quad 
    \tilde \Gamma_{2,3}(t) = 
    \begin{pmatrix}
      \sqrt{t} \\ \pm \frac{2}{\sqrt[4]{t}} 
    \end{pmatrix},\quad 
    \tilde \Gamma_{4,5}(t) = 
    \begin{pmatrix}
      -\sqrt{t} \\ \pm\frac{2i}{\sqrt[4]{t}} 
    \end{pmatrix},\quad 
    \tilde \Gamma_{6,7} =
    \begin{pmatrix}
      -3 \\ \pm \frac{\sqrt{6-2t}}{3}
    \end{pmatrix}
  \]
  in the local coordinates $(z,s_1)$ of $\tilde X$. 
  They are precisely taken to the corresponding branches $\Gamma_i(t)$ from 
  Example \ref{exp:Morsification} by $\nu$. Again, only the first five of them have 
  limit points close to $\nu^{-1}(\{0\})$ for $t \to 0$, i.e. only the first five branches 
  contribute to $\Eu^{\D f}(X,0)$ for sufficiently small $\varepsilon \gg \eta >0$. 
  Therefore, 
  \[
    \Eu^{\D f}(X,0) = 5 = \mu_f(2;X,0), 
  \]
  as anticipated.
\end{example}

\begin{remark}
  Definition \ref{def:MultiplicityOfOneForms} 
  and Proposition \ref{prp:MorsificationLemmaForOneForms} suggest yet another interpretation 
  of the numbers $\mu^\omega(\alpha;X,p)$, namely as \textit{microlocal intersection numbers}. 
  For a stratum $\mathscr S_\alpha$ of $X$ and its closure $X_\alpha$ one can define 
  \textit{conormal cycle} of $X_\alpha$ as 
  \[
    \Lambda_\alpha = \{ (\varphi,x) \in \Omega^1_U : x \in \mathscr S_\beta \subset X_\alpha, \quad 
    \varphi|T_x \mathscr S_\beta = 0 \}.
  \]
  This is a Whitney stratified subspace of the total space of the vector bundle $\Omega^1_U$. 
  The Whitney conditions imply that the fundamental class $[\Lambda_\alpha] \in H_{2n}^{BM}(U)$ 
  is a well defined cycle in Borel-Moore homology. So is the class $[\sigma(\omega)]$ of 
  the section $\sigma(\omega)$ of $\omega$ on $U$. In this context, 
  Proposition \ref{prp:MorsificationLemmaForOneForms} appears as a moving lemma, which puts 
  the two cycles in a general position. Clearly, the number of intersection points of 
  $[\Lambda_\alpha]$ and $[\sigma(\omega)]$ coincides with 
  $\mu^\omega(\alpha;X,p) = \Eu^\omega(X_\alpha,p)$.
  See also \cite[Corollary 5.4]{BrasseletMasseyParameswaranSeade04}.
\end{remark}

\section{The Euler obstruction as a homological index}
\label{sec:TheMilnorNumberAsAHomologicalIndex}

Throughout this section let again $U \subset \CC^n$ be an open 
domain and 
$X \subset U$ a closed, equidimensional, reduced, complex analytic space. 

For a holomorphic function $f \colon U \to \CC$ with an 
isolated singularity on $X$ at a point $p\in X$,
Proposition \ref{prp:MuEqualsEu} and Corollary \ref{cor:MuEqualsEuForFunctions}
suggest the following 
interpretation of the Euler obstruction: 
In a morsification $F = (f_t,t)$ of $f|(X,p)$ the singularities 
of $f|(X,p)$ become Morse critical points on the regular 
strata $\mathscr S_\alpha$.
In this sense, a morsification 
separates the singularities of the function $f|(X,p)$ from 
the singularities of the space $(X,p)$ itself. 
The Euler obstructions $\Eu^{\D f}\left(X_\alpha,p\right)$ of $\D f$ on the 
closures $X_\alpha = \overline{\mathscr S}_\alpha$ of the strata know 
the outcome of this separation beforehand and even without 
a given concrete morsification. 
A particular, but remarkable consequence of 
these considerations is that $\Eu^{\D f}(X_\alpha,p) = 0$ for all $\alpha \in A$ 
whenever 
$f$ does not have a singularity on $(X,p)$ -- independent 
of the singularities of the germ $(X,p)$ itself.

Suppose for the moment that also the space $(X,p)$ has itself only an isolated singularity 
so that the homological index $\Ind_{\hom}(\D f, X,p)$ as in 
\cite{EbelingGuseinZadeSeade04} is defined.
The comparison of $\Eu^{\D f}(X,p)$ with $\Ind_{\hom}(\D f, X, p)$ 
is based on 
the fact that both the Euler obstruction and the homological 
index satisfy 
the law of conservation of number and 
that they coincide at Morse critical 
points. In an arbitrary unfolding $F = (f_t,t)$ of 
$f|(X,p)$ we can therefore use both the Euler obstruction 
and the homological index to count the number of 
Morse critical points on $X_{\reg}$ arising from $f|(X,p)$.
But for a fixed unfolding parameter $t = \eta$
only the Euler obstruction $\Eu^{\D f_\eta}(X,p)$ can be 
used to measure whether $f_\eta$ is still singular at $(X,p)$ 
or whether all singularities of $f$ have left from the point $p$ 
for $t = \eta\neq 0$.
If the latter is the case -- as for example in a morsification -- 
the homological index $\Ind_{\hom}(\D f_{\eta},X,p)$ is 
\[
  \Ind_{\hom}(\D f_\eta,X,p) = 
  \Ind_{\hom}(\D f,X,p) - 
  \Eu^{\D f}(X,p) =
  - k'(X,p).
\]
The number $k'(X,p)$ is an invariant of the space $(X,p)$, but unknown in general. 
Therefore, the homological index $\Ind_{\hom}(\D f,X,p)$ can not be 
used to count the number of Morse critical points on $X_{\reg}$ in a morsification; 
it only seperates the singularities 
of the function $f$ from the singularities of $X$ up to an unknown quantity.

We return to the more general setting of an arbitrarily singular $X \subset U$.
Suppose $\omega$ is a holomorphic $1$-form on $U$ and let $p\in X$ be a point 
for which $\omega$ has an isolated zero 
on $(X,p)$. Then $\Eu^\omega(X_\alpha,p)$ is counting the number of 
simple zeroes 
on $\mathscr S_\alpha$ 
close to $p$
in a generic perturbation 
$\omega_\eta$ of $\omega$. It is evident from the construction that 
we may restrict our attention to the case where $X = X_\alpha = \overline{\mathscr S}_\alpha$ 
is irreducible and reduced and we only need to consider 
isolated zeroes of $\omega_\eta$ on $X_{\reg}$. 
Translating the previous discussion to this setting 
we see that -- conversely -- a homological index $I(\omega,X,p)$ 
has to coincide with the Euler obstruction $\Eu^\omega(X,p)$ whenever the following 
two conditions are met:
\begin{enumerate}[1.]
  \item $I(\omega,X,p)$ coincides with $\Eu^\omega(X,p)$ 
    whenever $p \in X$ is a smooth point of $X$.
  \item For every singular point $p$ of $X$ one has 
    \[
      I(\omega,X,p) = 0
    \]
    whenever $\omega$ is a $1$-form such that $\omega|(X,p)$ 
    is nonzero or has at most a simple zero at $p$ in the stratified sense.
\end{enumerate}
It is therefore worthwhile to investigate once again the 
structural reasons as to why 
1.\! is satisfied for $\Ind_{\hom}(\omega,X,p)$ 
at smooth points and why $\Eu^{\omega}(X,p) = 0$
whenever $\omega$ has at most a simple zero on $X$ at a point 
$p$ on a lower dimensional stratum.
We will exploit these 
reasons for the construction of a homological index 
$I(\omega,X_\alpha,p)$ which satisfies 1. and 2. simultaneously.

\medskip

The fact that the homological index of a $1$-form $\omega$ with an isolated 
zero at a smooth point $(X,p) \cong (\CC^n,p)$ coincides with its Euler 
obstruction and its topological index is based on the following fact. 
In local coordinates $x_1,\dots,x_n$ of $(X,p)$, the complex 
(\ref{eqn:ComplexWedgeWithOmega}) becomes a \textit{Koszul complex} 
on the local ring $\OO_{X,p}$ in the components 
of $\omega = \sum_{i=1}^n \omega_i \D x_i$. 
Since $\OO_{X,p}$ is Cohen-Macaulay and the zero locus of 
$\omega$ is isolated, the $\omega_i$ must form a regular sequence 
on $\OO_{X,p}$ and the following lemma applies,
cf. \cite[Corollary 1.6.19]{BrunsHerzog93}.

\begin{lemma}
  \label{lem:KoszulComplex}
  Let $(R, \mathfrak m)$ be a Noetherian local ring, $M = R^r$ a free 
  module, $v = \left(v_1,\dots,v_r\right)^T \in M$ an element and 
  \begin{equation}
    \label{eqn:KoszulKomplexOfAnElement}
    K^\bullet(v,R):
    \xymatrix{
      0 \ar[r] & 
      R \ar[r]^{v \wedge} & 
      M \ar[r]^{v \wedge} & 
      \bigwedge^2 M \ar[r]^{v \wedge} & 
      \cdots \\
      & & \cdots \ar[r]^{v \wedge} & 
      \bigwedge^{r-1} M \ar[r]^{v \wedge} & 
      \bigwedge^{r} M \ar[r] & 
      0
    }
  \end{equation}
  the Koszul complex associated to $v$. We consider $R = \bigwedge^0 M$ to be 
  situated in degree zero, $M = \bigwedge^1 M$ in degree one, etc.
  \begin{enumerate}[i)]
    \item Whenever $(v_1,\dots,v_r)$ is a regular sequence on 
      $R$ as an $R$-module, then (\ref{eqn:KoszulKomplexOfAnElement}) is 
      exact except for the last step where we find 
      \[
	H^r\left( K^\bullet(v,R) \right) = R / \langle v_1,\dots,v_r\rangle.
      \]
    \item Whenever $v \notin \mathfrak m M$, the Koszul complex is exact. 
  \end{enumerate}
\end{lemma}

Consequently, $\Ind_{\hom}(\omega,X,p) = \dim_{\CC} \OO_{X,p}/\langle \omega_1,\dots,\omega_n\rangle$
which evaluates to $1$ on simple zeroes of $\omega$.
Part ii) of this lemma explains why the homological index of $\omega$ is zero at all smooth 
points $q \in X$ where $\omega$ does not vanish.

\medskip

From this viewpoint, the difficulty in comaring the Euler obstruction 
of a $1$-form $\omega$ at a \textit{singular} point $p$ of $X$ with its 
homological index at $p$ stems from the fact that the restriction 
$\omega|(X,p)$ is not anymore an element of a free module, but 
of the module of K\"ahler differentials $\Omega_{X,p}^1$.
The key idea is to address this issue by replacing $\Omega^1_{X,p}$ 
and $\omega$ with the Nash bundle $\tilde \Omega^1$ and the section 
$\nu^* \omega$. In order to work with finite $\OO_X$-modules we need 
to consider the derived pushforward of the associated bundles.
Analogous to Lemma \ref{lem:KoszulComplex} ii) we find 
the following.

\begin{lemma}
  \label{lem:DerivedPushforwardAtRegularPoints}
  Let $U \subset \CC^n$ be an open domain, $X\subset U$ an irreducible and reduced 
  closed 
  analytic subspace of dimension $d$, and $\nu \colon \tilde X \to X$ its 
  Nash modification.
  For any point $p \in X$
  the stalk at $p$ of the complex of sheaves 
  \[
    \mathbb R \nu_* \left( \tilde \Omega^\bullet, \nu^* \omega \wedge - \right)_p
  \]
  is exact, whenever $\omega$ does not annihilate any limiting tangent space 
  $V$ from $X_{\reg}$ at $p$. 
\end{lemma}

\begin{proof}
  The statement that $\omega$ does not annihilate any limiting tangent space $V$ 
  of a top-dimensional stratum at $p$ is equivalent to saying that 
  $\nu^{*}\omega$ is nonzero at every point $(p,V) \in \tilde X$ 
  in the fiber $\nu^{-1}(\{p\})$ of the Nash modification over $p$. 

  If $\nu^* \omega$ is nonzero then, according to Lemma 
  \ref{lem:KoszulComplex} ii), the complex of sheaves 
  \begin{equation}
    \xymatrix{ 
      0 \ar[r] & 
      \OO_{\tilde X} \ar[r]^{\nu^* \omega \wedge} & 
      \tilde \Omega^1 \ar[r]^{\nu^* \omega \wedge} & 
      \tilde \Omega^2 \ar[r]^{\nu^* \omega \wedge} & 
      \cdots \ar[r]^{\nu^* \omega \wedge} & 
      \tilde \Omega^{d-1} \ar[r]^{\nu^* \omega \wedge} & 
      \tilde \Omega^d \ar[r] & 
      0
    }
    \label{eqn:NashBundleComplex}
  \end{equation}
  is exact along $\nu^{-1}( \{ p\})$ and therefore  
  quasi-isomorphic to the zero complex. Consequently, also the stalk at $p$ of the 
  derived pushforward of this complex has to vanish.
\end{proof}

\begin{theorem}
  \label{thm:MainTheorem}
  Suppose $U \subset \CC^n$ is an open domain, $X \subset U$ a reduced, equidimensional
  complex analytic subspace of dimension $d$, $S = \{ \mathscr S_\alpha\}_{\alpha \in A}$ 
  a complex analytic stratification satisfying Whitney's condition A, 
  and $\omega$ a holomorphic $1$-form with 
  an isolated zero on $X$ in the stratified sense at a point $p$. Then 
  \begin{equation}
    \Eu^\omega(X,p) = 
    (-1)^{d} 
    \chi\left( \RR\nu_* \left( \tilde \Omega^\bullet, \nu^* \omega \wedge - \right)_p \right) 
    \label{eqn:MainFormula}
  \end{equation}
  where $\nu \colon \tilde X \to X$ is the Nash modification and 
  $(\tilde \Omega^\bullet, \nu^*\omega\wedge - )$ is the complex 
  of coherent sheaves on $\tilde X$ given by the exterior powers 
  of the Nash bundle and multiplication with $\nu^*\omega$.
\end{theorem}

\begin{corollary}
  \label{cor:MainCorollary}
  Let $(X,p) \subset (\CC^n,p)$ be a 
  reduced complex analytic space with a
  complex analytic Whitney stratification
  $S = \{\mathscr S_\alpha\}_{\alpha \in A}$.
  Suppose 
  \[
    f : (\CC^n,p) \to (\CC,0)
  \]
  is a holomorphic 
  function with an isolated singularity on $(X,p)$. 
  For $\alpha \in A$ let 
  $\nu \colon \tilde X_\alpha \to X_\alpha$ be the Nash modification 
  of the closure $X_\alpha = \overline{ \mathscr S_\alpha }$ and 
  $\tilde \Omega^k_\alpha$ the $k$-th exterior power of the dual of the 
  Nash bundle on $\tilde X_\alpha$. 
  Then 
  \begin{equation}
    \mu_f(\alpha;X,0) = 
    (-1)^{d(\alpha)} 
    \chi\left( \RR\nu_* \left( \tilde \Omega_\alpha^\bullet, \nu^* \D f \wedge - \right)_p \right).
    \label{eqn:MainFormulaForFunctions}
  \end{equation}
\end{corollary}

\begin{proof}
  We may apply Theorem \ref{thm:MainTheorem} to the space $X_\alpha = \overline{\mathscr S_\alpha}$ 
  and the restriction of the $1$-form $\D f$ to it.
\end{proof}

\begin{proof}(of Theorem \ref{thm:MainTheorem})
  The sheaves in the complex 
  $\mathbb R \nu_* (\tilde \Omega^\bullet, \nu^* \omega \wedge - )$ are 
  finite $\OO_n$-modules since the morphism $\nu$ is proper. 
  By assumption, $\omega$ has an isolated zero on $(X,p)$ in the stratified 
  sense and hence 
  Lemma \ref{lem:DerivedPushforwardAtRegularPoints} implies that the cohomology 
  of this complex is supported at the origin. In particular, its Euler 
  characteristic is finite. 

  Suppose $W = (\omega_t,t)$ is an unfolding of $\omega|(X,p)$ as in Proposition 
  \ref{prp:MorsificationLemmaForOneForms} and -- possibly after shrinking $U$ -- let 
  \[
    W \colon U \times T \to \CC^n \times T
  \]
  be a suitable representative thereof. Denote by $\pi \colon U \times T \to T$ the 
  projection to the parameter $t$. The unfolding of $\omega$ 
  induces a family of complexes of sheaves 
  $\left( \tilde \Omega^\bullet, \nu^* \omega_t \wedge -\right)$ on the 
  Nash transform $\tilde X$ and hence also on the derived pushforward. 
  This furnishes a complex of coherent sheaves  
  \[
    \mathbb R \nu_* \left( \tilde \Omega^\bullet, \nu^* \omega_t \wedge - \right)
  \]
  on $U \times T$ which becomes a family of complexes over $T$ via the projection $\pi$. 
  Clearly, every sheaf $R^k \nu_* \tilde \Omega^r$ is $\pi$-flat. We may apply 
  the main result of \cite{GomezMont98}: There exist neighborhoods $p \in U' \subset U$ 
  and $0 \in T' \subset T$ 
  such that for every $\eta \in T'$ we have 
  \begin{equation}
    (-1)^d \chi\left( 
    \mathbb R \nu_* \left( \tilde \Omega^\bullet, \nu^* \omega_0 \wedge - \right)_p
    \right)
    =
    (-1)^d \sum_{x \in U'} 
    \chi \left(
    \mathbb R \nu_* \left( \tilde \Omega^\bullet, \nu^* \omega_\eta \wedge - \right)_x
    \right),
    \label{eqn:LawOfConservationOfNumberForLocalEulerCharacteristics}
  \end{equation}
  i.e. the Euler characteristic satisfies the law of conservation of number. 

  Suppose $U', T'$ and $B_\varepsilon$ have also been chosen as in Proposition 
  \ref{prp:MorsificationLemmaForOneForms} and fix $\eta \in T'$, $\eta \neq 0$. 
  By construction, $\omega_\eta$ has only 
  simple, isolated zeroes on the interior of $X \cap B_\varepsilon$ and none 
  on the boundary. 
  
  Whenever $x \in (X \setminus X_{\reg}) \cap B_\varepsilon$ 
  is such a point, at which $\omega_\eta$ has a 
  simple zero \textit{outside} $X_{\reg}$, 
  the restriction of $\omega_\eta$ to any limiting tangent 
  space $V$ of $X_{\reg}$ at $x$ is nonzero and consequently
  \[
    \mathbb R\nu_* \left( \tilde \Omega, \nu^* \omega \wedge - \right)_x \cong_{\qis} 0
  \]
  according to Lemma \ref{lem:DerivedPushforwardAtRegularPoints}.

  Whenever $x \in X_{\reg} \cap B_\varepsilon$ is a point with a 
  simple zero of $\omega_\eta$ at $x$ we find the following. 
  The Nash modification $\nu$ is a local isomorphism around $x$ 
  and therefore 
  \[
    \mathbb R \nu_* \left( \tilde \Omega^\bullet, \nu^* \omega_\eta \wedge - \right)_x 
    \cong \left( \Omega^\bullet_{X,x}, \omega_\eta \wedge - \right)
  \]
  is the Koszul complex on the modules $\Omega^k_{X,x}$. Lemma 
  \ref{lem:KoszulComplex} allows us to compute the Euler characteristic 
  \[
    (-1)^d\cdot  \chi\left( \Omega_{X,x}^\bullet, \omega_\eta \wedge - \right) = 1.
  \]
  The statement now follows from the principle of conservation of number. 
\end{proof}

\begin{example}
  \label{exp:PresentationOfComplexes}
  We continue with Example \ref{exp:EulerObstructionOfAOneForm}. As previously discussed, 
  the only interesting stratum of $X$ is $\mathscr S^2= X_{\reg}$. To prepare for the computations 
  of $\mu(2;f,0)$ we will describe a complex of graded $S$-modules representing 
  $(\tilde \Omega^\bullet,\nu^* \D f\wedge -)$. We set 
  $A = \CC[x,y,z]$, $S= A[s_0,s_1,s_2]$ and consider $S$ as a homogeneous coordinate 
  ring of $\PP^2_A$ over $A$. The ideal $J\subset S$ of homogeneous equations for 
  the Nash transform $\tilde X$ is obtained from the equations for the total transform 
  by saturation: Denote by $L$ the ideal of $2\times 2$-minors of the matrix
  \[
    \begin{pmatrix}
      s_0 & s_1 & s_2 \\
      \frac{\partial h}{\partial x} &
      \frac{\partial h}{\partial y} &
      \frac{\partial h}{\partial z} 
    \end{pmatrix}.
  \]
  Over $X_{\reg}$ these equations describe the graph of the rational 
  map $\Phi$ underlying the Nash blowup (\ref{eqn:RationalMapForNashBlowup}).
  Now
  \[
    J = (\left\langle h \right\rangle + L) : \langle y,z \rangle^\infty,
  \]
  where $\langle y,z\rangle$ is the ideal defining the singular locus of $X$ on which 
  $\Phi$ 
  is not defined.

  Let $Q^p$ be the module representing $\bigwedge^p \mathcal Q$ 
  with $\mathcal Q$ the tautological quotient bundle on $\PP^2_A$. A graded, free resolution of 
  the $Q^p$ is given by appropriate shifts of the Koszul complex in the $s$-variables. 
  Let 
  \[
    \theta = s_0 \cdot e_0 + s_1 \cdot e_1 + s_2 \cdot e_2 \in 
    H^0(\PP^2_A,\OO(1)^3) \cong \left( S^3 \right)_1 
  \]
  be the tautological section.
  Together with 
  \[
    \nu^* \D f = 
    -2(x-z) \cdot e_0 + 2y \cdot e_1 + 2(x-z) \cdot e_2 \in H^0(\PP^2_A, \OO^3) \cong \left(S^3\right)_0
  \]
  we obtain the following double complex.
  \[
    \xymatrix{
      & 0 & 0 & 0 & \\
      0 \ar[r] & 
      Q^0 \ar[r]^{\nu^* \D f\wedge} \ar[u] & 
      Q^1 \ar[r]^{\nu^* \D f\wedge} \ar[u] & 
      Q^2 \ar[r] \ar[u] & 
      0 \\
      0 \ar[r] & 
      \bigwedge^0 S^3 \ar[r]^{\nu^* \D f\wedge} \ar[u]^\varepsilon & 
      \bigwedge^1 S^3 \ar[r]^{\nu^* \D f\wedge} \ar[u]^\varepsilon & 
      \bigwedge^2 S^3 \ar[r] \ar[u]^\varepsilon & 
      0 \\
      & 0\ar[u]\ar[r] & 
      \left(\bigwedge^0 S^3 \right) \otimes S(-1) 
      \ar[u]^{\theta \wedge}\ar[r]^{\nu^* \D f \wedge} & 
      \left(\bigwedge^1 S^3 \right) \otimes S(-1) \ar[u]^{\theta \wedge} \ar[r] & 
      0 \\
      & & 0 \ar[r] \ar[u] & 
      \left(\bigwedge^0 S^3 \right) \otimes S(-2) 
      \ar[u]^{\theta \wedge}\ar[r] & 
      0 \\
      & & & 0 \ar[u] & \\
    }
  \]
  For every $q$ the module $M^q$ representing the restriction $\bigwedge^q \tilde \Omega^1$ 
  of $\mathcal Q^q$ to $\tilde X$ is given 
  by $Q^q \otimes S/J$. The complex of sheaves $(\tilde \Omega^\bullet, \nu^*\D f\wedge -)$ 
  on $\tilde X$ is thus represented by the complex of graded modules 
  \[
    \left( M^\bullet,\nu^*\D f\wedge - \right) = 
    \left( Q^\bullet \otimes S/J, \nu^* \D f\wedge - \right).
  \]
  As we shall see in the next section, 
  Proposition \ref{prp:DerivedPushforwardOfComplexes}, 
  we can compute the derived pushforward $\mathbb R\nu_*(\tilde \Omega^\bullet, \nu^* \D f\wedge -)$ 
  via a truncated \v{C}ech-double-complex on the complex of 
  modules $(M^\bullet, \nu^* \D f \wedge - )$. 
\end{example}

\section{How to compute $\mu_f(\alpha;X,0)$ for $\overline{\mathscr S_\alpha}$ a hypersurface}
\label{sec:Algorithm}

The following section will be phrased in purely algebraic 
terms. This is due to the fact that the complex numbers 
are not a computable field and also the ring of convergent 
power series is usually not available in computer algebra 
systems for symbolic computations. 
If we were working in the projective setting, Chow's theorem 
\cite{Chow49}
and the GAGA-principles due to Serre \cite{Serre55} allow us
to restrict to the algebraic case. In the local context 
we can not do so.
For these reasons, we will assume that both $(X,0) \subset (\CC^n,0)$ 
and either $f$ or $\omega$ as in Theorem \ref{thm:MainTheorem} 
or Corollary \ref{cor:MainCorollary} are algebraic and defined 
over some finite extension field $K$ of $\QQ$. 

Thus we will -- with a view towards Theorem \ref{thm:MainTheorem} --
work with proper maps 
$\pi \colon X \to Y$
of algebraic spaces. Let $\mathcal F$ be a coherent 
algebraic sheaf on $X$ and $\mathcal F^h$ its 
analytification.
It is well known that the sheaves $R^p\pi_*(\mathcal F)$ 
are $\OO_Y$-coherent. 
Grauert's theorem on direct images \cite{Grauert60} assures 
that also the direct images $R^p\pi_*(\mathcal F^h)$ are 
$\mathcal O_Y^h$-coherent and using Cech cohomology we obtain 
a natural morhism of cohomology sheaves
\[
  \varepsilon \colon R^p\pi_*(\mathcal F)\to R^p\pi_*(\mathcal F^h).
\]
for every $p$. 

We will see below that whenever $\pi$ is the restriction of a projection 
\[
  \pi' \colon \PP^r \times (\CC^n,0) \to (\CC^n,0),
\]
as we may assume for the purpose of this article by virtue 
of the Pl\"ucker embedding, one can express the direct images 
of a coherent algebraic sheaf $\mathcal F$ in terms of the 
cohomology of the relative twisting sheaves $\OO(-w)$ and 
vice versa for their analytifications. Now the formal 
completions of the rings 
\[
  \CC\{x_1,\dots,x_n\} \quad \text{ and } \quad \CC[x_1,\dots,x_n]_{\langle x_1,\dots, x_n \rangle}
\]
are isomorphic and so are the formal completions of 
\[
  R^p\pi'_*(\OO(-w))\quad \text{ and }\quad R^p\pi'_*(\OO^h(-w))
\]
for all $p$ and $w$. In what follows, the sheaf $\mathcal F$ 
-- or, more generally, the complex of sheaves $\mathcal F^\bullet$ -- 
will always 
have $\mathbb R \pi'_*(\mathcal F)$ and $\mathbb R \pi'_*(\mathcal F^h)$ 
with isolated support at the origin. Thus, their Euler characteristics 
both have to coincide with the Euler characteristic of their isomorphic 
formal completions. In particular, the comparison morphism 
$\varepsilon$ above is an isomorphism in this case and we may therefore 
carry out all computations in the algebraic setting. 

\medskip

Let $A$ be a commutative Noetherian ring. We set $S = A[s_0,\dots,s_{r}]$ 
and consider $S$ as a graded $A$-algebra. On the geometric side let 
\[
  \pi \colon \PP^r_A \to \Spec A
\]
be the associated projection.
Let $\OO= \tilde S$ be the structure sheaf of $\PP^r_A$ and $\OO(-w)$ the 
relative twisting sheaves for $w \in \ZZ$. 
Given any finitely generated graded $S$-module $M$ there is a corresponding 
sheaf $\tilde M$ of $\OO$-modules on $\PP^r_A$. 
We will first describe how to compute $\RR \pi_* (\tilde M)$ as a complex of 
finitely generated 
$A$-modules up to quasi-isomorphism 
and then generalize these results for complexes of finite, graded $S$-modules
$(M^\bullet,D^\bullet)$ and their associated complexes of sheaves 
on $\PP^r_A$. 

We may use \v{C}ech cohomology with respect to the 
canonical open covering of $\PP^r_A$.
For a graded $S$-module $M$ let 
\[
  \check{C}^p( M ) = 
  \bigoplus_{0 \leq i_0<i_1<\cdots<i_p\leq r} 
  M \otimes S[(s_{i_0} s_{i_1}\cdots s_{i_p})^{-1}].
\]
These modules are not finitely generated over $S$, but they have a
natural structure as a direct limit of finite $S$-modules given 
by the submodules 
\[
  \check{C}^p_{\leq d}(M) = 
  \bigoplus_{0 \leq i_0<i_1<\cdots<i_p\leq r} 
  M \otimes \frac{1}{(s_{i_0} s_{i_1}\cdots s_{i_p})^d}.
\]
The \v{C}ech-complex of twisted sections in $\tilde M$ is 
obtained from the $\check{C}^p(M)$ together with the 
differential $ \check{\D}\, \colon \check{C}^p(M) \to \check{C}^{p+1}(M)$ taking 
an element 
\[
  \frac{a_{i_1,\dots,i_p} }{(s_{i_0}\cdots s_{i_p})^{d}},
\]
$a_{i_1,\dots,i_p} \in M$ to the element in $\check{C}^{p+1}(M)$ with 
component $(j_{0},\dots,j_{p+1})$ given by
\[
  \frac{1}{(s_{j_0} \cdots s_{j_{p+1}})^d}
  \sum_{k=0}^{p+1} (-1)^k s_{j_k}^d a_{j_0,\dots,\hat{j_k},\dots,j_{p+1}}.
\]
As usual, $\hat{\cdot}$ indicates that the index is to be omitted.
We will write 
\[
  \check H^p(M) := H^p\left( \check{C}^\bullet(M) \right)
  \quad \text{ and } 
  \quad 
  \check H_{\leq d}^p(M) = H^p\left( \check{C}_{\leq d}^\bullet(M) \right)
\]
for the $p$-th cohomology of the \v{C}ech complex on a module $M$ 
and its truncations.

The modules $S(-w)$ and the corresponding twisting sheaves 
$\mathcal O(-w)$ have a well known cohomology, 
see \cite[Chapter III.5]{Hartshorne77}.
We deliberately identify 
\[
  S(-w) = \bigoplus_{d \in \ZZ} R^0 \pi_*( \mathcal O(d-w))
\]
and set 
\[
  E(-w) =
  \bigoplus_{d \in \ZZ} R^r\pi_*(\OO(d-w)) \cong
  \check{H}^r\left(S(-w))\right).
\]
The last term has a structure as a direct limit of $S$-modules via the maps
\begin{eqnarray*}
  \Psi_d \colon S(d(r+1)-w)/\langle s_0^d,\dots,s_r^d\rangle &
  \overset{\cong}{\longrightarrow}&
  \check{H}^r_{\leq d} (S(-w)) \subset 
  \check{H}^r\left( S(-w) \right), \\ 
  1 &\mapsto& \frac{1}{(s_0\cdots s_r)^d}.
\end{eqnarray*}
The pairing of monomials
\begin{eqnarray*}
  S(w) \times E(-w-r-1) & \to & A \\
  \left( s_0^{\alpha_0} s_1^{\alpha_1} \cdots s_r^{\alpha_r}, 
  \frac{1}{s_0^{\beta_0} s_1^{\beta_1}\cdots s_r^{\beta_r}} \right) &\mapsto&
  \begin{cases}
    1 & \text{ if } \alpha_i = \beta_i -1 \quad \forall i \\
    0 & \text{ otherwise }
  \end{cases}
\end{eqnarray*}
provides us with an identification 
\begin{equation}
  E(-w-r-1) \cong \Hom_{A}(S(w),A)
  \label{eqn:SModuleStructureOnHr}
\end{equation}
for all $w \in \ZZ$. Note that this pairing is compatible with the natural 
$S$-module structure on both sides. 

\begin{proposition}
  \label{prp:DerivedPushForwardSingleModule}
  Let $M$ be a graded $S$-module and 
  \[
    K^\bullet : 
    \xymatrix{ 
      0 & 
      M \ar[l] & 
      \oplus_{i_0=1}^{\beta_0}  S(-w_{0,i_0}) \ar[l]_{\varepsilon} & 
	\oplus_{i_{-1}=1}^{\beta_{-1}}  S(-w_{-1,i_{-1}}) \ar[l]_{D^0} & 
	\cdots \ar[l] \\
	& \cdots & 
	\oplus_{i_{-r}=1}^{\beta_{-r}}  S(-w_{-r,i_{-r}}) \ar[l]_{D^{-r+1}} &
	\oplus_{i_{-r-1}=1}^{\beta_{-r-1}}  S(-w_{-r-1,i_{-r-1}}) \ar[l]_{D^{-r}} &
    }
  \]
  an exact complex. Let 
  $\left( \bigoplus_{i_\bullet=1}^{r_\bullet} E(w_{\bullet,i_\bullet}), D^\bullet  \right)$ 
  be the complex with the $S$-module 
  $\bigoplus_{i_{-k}=1}^{r_{-k}} E(w_{-k,i_{-k}})$ as in (\ref{eqn:SModuleStructureOnHr})
  in cohomological degree $-k$ and 
  $D^k$ the differentials induced by the same differentials as those in $K^\bullet$.
  Then there is a short exact sequence
  \[
    \xymatrix{ 
      0 \ar[r] & 
      M \ar[r] & 
      \check H^0(M) \ar[r] & 
      H^{-r}\left( \bigoplus_{i_\bullet=1}^{\beta_\bullet} E(-w_{\bullet,i_\bullet}), D^\bullet \right)
      \ar[r] &
      0
    }
  \]
  and isomorphisms 
  \[
    \check H^p(M) \cong 
    H^{p-r}\left( \bigoplus_{i_\bullet=1}^{\beta_\bullet} E(-w_{\bullet,i_\bullet}), D^\bullet \right)
  \]
  for $0 < p \leq r$.
\end{proposition}

\begin{proof}
  The statements follow from a diagram chase in the double complex 
  (\ref{eqn:CechDoubleComplex}). Note that in (\ref{eqn:CechDoubleComplex}) 
  all columns but the last one are exact by construction. 
  The same holds for all rows but the first one.
  Since taking cohomology commutes with direct sums, the complex 
  \[
    \left( \bigoplus_{i_\bullet=1}^{\beta_\bullet} 
    E(-w_{\bullet,i_\bullet}), D^\bullet \right)
  \] 
  is identical with the last 
  column of (\ref{eqn:CechDoubleComplex}), while the first row is the 
  \v{C}ech-complex on $M$.
  \begin{sidewaysfigure}
    \vspace{13cm}
    \begin{equation}
      \xymatrix{
	& 0 & 0 & 0 &   & 0 & 0 & \\
	0 \ar[r] &
	M \ar[r] \ar[u]& 
	\check{C}^0(M) \ar[r]^{\check{\D\,}} \ar[u] & 
	\check{C}^1(M) \ar[r]^{\check{\D\,}} \ar[u] & 
	\cdots \ar[r]^{\check{\D\,}} & 
	\check{C}^{r-1}(M) \ar[r]^{\check{\D\,}} \ar[u] & 
	\check{C}^r(M) \ar[u] \ar[r] &
	0\\
	0 \ar[r] & 
	K^0 \ar[r] \ar[u]^\varepsilon & 
	\check{C}^0(K^0) \ar[r]^{\check{\D\,}} \ar[u]^\varepsilon & 
	\check{C}^1(K^0) \ar[r]^{\check{\D\,}} \ar[u]^\varepsilon & 
	\cdots \ar[r]^{\check{\D\,}} &
	\check{C}^{r-1}(K^0) \ar[r]^{\check{\D\,}} \ar[u]^\varepsilon & 
	\check{C}^r(K^0) \ar[u]^\varepsilon \ar[r] & 
	\check{H}^r(K^0) \ar[r] \ar[u] &
	0 \\
	0 \ar[r] & 
	K^{-1} \ar[r] \ar[u]^D & 
	\check{C}^0(K^{-1}) \ar[r]^{\check{\D\,}} \ar[u]^D & 
	\check{C}^{1}(K^{-1}) \ar[r]^{\check{\D\,}} \ar[u]^D & 
	\cdots \ar[r]^{\check{\D\,}} &
	\check{C}^{r-1}(K^{-1}) \ar[r]^{\check{\D\,}} \ar[u]^D & 
	\check{C}^r(K^{-1}) \ar[u]^D \ar[r] & 
	\check{H}^r(K^{-1}) \ar[r] \ar[u]^D &
	0 \\
	& 
	\vdots \ar[u]^{D}&
	\vdots \ar[u]^{D}&
	\vdots \ar[u]^{D}&
	&
	\vdots \ar[u]^{D}&
	\vdots \ar[u]^{D}&
	\vdots \ar[u]^{D}&
	& \\
	0 \ar[r] & 
	K^{-r+1} \ar[r] \ar[u]^D & 
	\check{C}^0(K^{-r+1}) \ar[r]^{\check{\D\,}} \ar[u]^D & 
	\check{C}^{1}(K^{-r+1}) \ar[r]^{\check{\D\,}} \ar[u]^D & 
	\cdots \ar[r]^{\check{\D\,}} &
	\check{C}^{r-1}(K^{-r+1}) \ar[r]^{\check{\D\,}} \ar[u]^D & 
	\check{C}^r(K^{-r+1}) \ar[u]^D \ar[r] & 
	\check{H}^r(K^{-r+1}) \ar[r] \ar[u]^D &
	0 \\
	0 \ar[r] & 
	K^{-r} \ar[r] \ar[u]^D & 
	\check{C}^0(K^{-r}) \ar[r]^{\check{\D\,}} \ar[u]^D & 
	\check{C}^{1}(K^{-r}) \ar[r]^{\check{\D\,}} \ar[u]^D & 
	\cdots \ar[r]^{\check{\D\,}} &
	\check{C}^{r-1}(K^{-r}) \ar[r]^{\check{\D\,}} \ar[u]^D & 
	\check{C}^r(K^{-r}) \ar[u]^D \ar[r] & 
	\check{H}^r(K^{-r}) \ar[r] \ar[u]^D &
	0 \\
	0 \ar[r] & 
	K^{-r-1} \ar[r] \ar[u]^D & 
	\check{C}^0(K^{-r-1}) \ar[r]^{\check{\D\,}} \ar[u]^D & 
	\check{C}^{1}(K^{-r-1}) \ar[r]^{\check{\D\,}} \ar[u]^D & 
	\cdots \ar[r]^{\check{\D\,}} &
	\check{C}^{r-1}(K^{-r-1}) \ar[r]^{\check{\D\,}} \ar[u]^D & 
	\check{C}^r(K^{-r-1}) \ar[u]^D \ar[r] & 
	\check{H}^r(K^{-r-1}) \ar[r] \ar[u]^D &
	0 \\
      }
      \label{eqn:CechDoubleComplex}
    \end{equation}
  \end{sidewaysfigure}
\end{proof}

We can use Proposition \ref{prp:DerivedPushForwardSingleModule} to describe 
$\mathbb R\pi_* (\tilde M)$ as a complex of finite $A$-modules.
  Choose any 
  \[
    d \geq \max\{ w_{-k,i_{-k}} : 0 \geq -k \geq -r-1 \} - r
  \]
  and let
  \[
    \Psi^{-k}_d \colon 
    \bigoplus_{i_{-k} = 1}^{\beta_{-k}} \frac{S(d(r+1)-w_{-k,i_{-k}})}{
    \langle s_0^d,\dots,s_r^d\rangle} \hookrightarrow 
    \bigoplus_{i_{-k}=1}^{\beta_{-k}} E(-w_{-k,i_{-k}})
  \]
  be the inclusions of finite $S$-modules as before. The restriction on the choice of 
  $d$ assures that the degree zero part of every 
  $E(-w_{-k,i_{-k}})$ is fully contained in the image of 
  $\Psi^{-k}$. Consequently, the homomorphism of complexes in degree 
  zero
  \[
    \Psi^{\bullet}_d \colon 
    \left(\bigoplus_{i_{\bullet} = 1}^{\beta_{\bullet}} \frac{S(d(r+1)-w_{\bullet,i_{\bullet}})}{
    \langle s_0^d,\dots,s_r^d\rangle}, D^\bullet \right)_0 \overset{\cong}{\longrightarrow}
    \left( \bigoplus_{i_\bullet=1}^{\beta_\bullet} E(-w_{\bullet,i_\bullet}), 
    D^\bullet \right)_0 
  \]
  is an isomorphism of complexes of finite $A$-modules.

In other words, there is a short exact sequence of free finite $A$-modules
\[
  \xymatrix{ 
    0 \ar[r] & 
    M_0 \ar[r] & 
    R^0\pi_*( \tilde M) \ar[r] &
    H^{-r} \left(\bigoplus_{i_{\bullet} = 1}^{\beta_{\bullet}} \frac{S(d(r+1)-w_{\bullet,i_{\bullet}})}{
    \langle s_0^d,\dots,s_r^d\rangle}, D^\bullet \right)_0 \ar[r] &
    0
  }
\]
and isomorphisms 
\[
  R^p\pi_*(\tilde M) \cong 
  H^{p-r} \left(\bigoplus_{i_{\bullet} = 1}^{\beta_{\bullet}} S(d(r+1)-w_{\bullet,i_{\bullet}})/
    \langle s_0^d,\dots,s_r^d\rangle, D^\bullet \right)_0 
\]
for $0 < p \leq r$.

In terms of \v{C}ech-cohomology this implies the following. We may replace every 
\v{C}ech complex $\check{C}^\bullet(K^{-p})$ in (\ref{eqn:CechDoubleComplex}) 
by its truncation $\check{C}^\bullet_{\leq d}( K^{-p} )$ and restrict to 
the degree zero strands in each term. Another diagram chase reveals a quasi-isomorphism 
\begin{equation}
  \mathbb R \pi_*(\tilde M) \cong \check{C}^\bullet_{\leq d} (M)_0
  \label{eqn:DerivedPushforwardAsTruncatedCechComplex}
\end{equation}
as complexes of finite $A$-modules.

\begin{proposition}
  \label{prp:DerivedPushforwardOfComplexes}
  Let $M^\bullet$ be a bounded complex of finitely generated, graded 
  $S$-modules and $K^{\bullet,q} \overset{\varepsilon}{\longrightarrow} M^q$ 
  a graded free resolution of every $M^q$ with 
  \[
    K^{-p,q} = \bigoplus_{i_{-p,q}=1}^{\beta_{-p,q}} 
    S\left( -w_{-p,q,i_{-p,q}} \right).
  \]
  Choose $d \geq \max\left\{ w_{-p,q,i_{-p,q}} : M^q \neq 0, -p > -r-1 \right\}-r$. 
  Then $\mathbb R \pi_*( \tilde M^\bullet )$ is quasi-isomorphic to the degree
  zero part of the total 
  complex of the double complex $C_{\leq d}^{\bullet,\bullet}$ with terms 
  $C_{\leq d}^{p,q} = \check{C}^{p}_{\leq d}( M^q)$:
  \[
    \mathbb R \pi_*( \tilde M^\bullet ) \cong_{\operatorname{qis}}
    \operatorname{Tot}\left( \check{C}^{\bullet}_{\leq d}(M^\bullet) \right)_0.
  \]
\end{proposition}

\begin{proof}
  The right derived pushforward of a single sheaf $\tilde M$ on $\PP^r_A$ is usually 
  defined via injective resolutions of $\tilde M$ and it is well known that 
  the resulting complex is quasi-isomorphic to the Cech-complex on $\tilde M$ 
  for the affine covering above. For a complex of sheaves $\tilde M^\bullet$ 
  the derived pushforward can be computed as the total complex of a double complex 
  $I^{\bullet,\bullet}$ of injective sheaves which forms an injective resolution 
  of $\tilde M^\bullet$. There is a corresponding spectral sequence 
  identifying this total complex with the total complex of the Cech-double complex for 
  $\tilde M^\bullet$ up to quasi isomorphism analogous to the case of 
  a single sheaf. The result now follows from 
  (\ref{eqn:DerivedPushforwardAsTruncatedCechComplex}): On the first page of 
  the spectral sequence of the Cech-double complex $\check C^\bullet(M^\bullet)$ 
  we may replace each term 
  $H^p( \check C^\bullet(M^q))$ by the truncation 
  $H^p( \check C^\bullet_{\leq d}(M^\bullet))$.
\end{proof}

We conclude with a brief description of how to use Proposition 
\ref{prp:DerivedPushforwardOfComplexes} in order to compute 
(\ref{eqn:MainFormula}). Let $(X,0) \subset (\CC^n,0)$ be a 
reduced algebraic hypersurface defined over some finite extension 
$K$ of $\QQ$ and 
$\omega$ an algebraic $1$-form as in Theorem \ref{thm:MainTheorem} 
defined over the same field.

Set $A = K[x_1,\dots,x_n]_{\langle x_1,\dots,x_n}\rangle$ and let
$S = A[s_1,\dots,s_r]$ be the homogeneous ring in the $s$-variables.
Let $J$ be the homogeneous ideal of $S$ defining 
the Nash transform $\tilde X \subset \PP^{n-1}\times (\CC^n,0)$ 
and $M^q$ the graded modules presenting the duals of the 
exterior powers of the Nash bundle $\tilde \Omega^q$ on $\tilde X$ 
together with the morphisms given by the pullbacks $\nu^*\omega$ 
as in Example \ref{exp:PresentationOfComplexes}.
\begin{enumerate}[1)]
  \item We can compute a partial graded free resolution of every one of the 
    $M^q$ using Gr\"obner bases and a mixed ordering whose first block is 
    graded and global in the $s$-variables and whose second block is 
    local in the $x$-variables. 
  \item From this we obtain the bound $d$ on the pole order for the Cech-double complex 
    and we can build the truncated Cech-double complex $\check C^\bullet_{\leq d}(M^\bullet)$ 
    as a double complex of finite $S$-modules. 
  \item The degree-$0$-strands of $\check C^\bullet_{\leq d}(M^\bullet)$ are finite 
    $A$-modules generated by monomials in the $s$-variables. We can choose generators 
    and relations accordingly and extract the induced matrices for $\nu^*\omega\wedge -$ over 
    $A$ from the maps defined over $S$. 
  \item Since $\omega$ had an isolated zero, the cohomology of the resulting 
    complex must be finite over $K$. We can proceed by the usual Groebner basis 
    methods for the computation of Euler characteristics.
\end{enumerate}
These computations apply in particular to the case 
$X = \overline{\mathscr S_\alpha}$ and
$\omega = \D f$ 
as in Corollary \ref{cor:MainCorollary}.

\begin{remark}
  Note that in Proposition \ref{prp:DerivedPushforwardOfComplexes}
  we do not need to compute a graded free resolution of 
  the whole complex $M^\bullet$ by means of a double complex of free, 
  graded $S$-modules, but only resolutions of the individual terms 
  $M^q$. With a view towards the application of Proposition 
  \ref{prp:DerivedPushforwardOfComplexes} for the computation of 
  (\ref{eqn:MainFormula}) this entails that the number $d$ can be chosen once and 
  for all for a given space $(X,0)$ and then used for every $1$-form 
  $\omega$ with isolated zero on $(X,0)$. 
\end{remark}

\bibliographystyle{alpha}
\bibliography{sources}

\end{document}